\documentclass[10pt,a4paper]{article}

\usepackage{amsmath,amsfonts,amssymb}
\usepackage{bbm}
\usepackage{cases}

\usepackage{kotex}
\usepackage{lipsum}
\usepackage{subcaption}

\usepackage{xcolor}
\usepackage{graphicx}
\usepackage{ulem}

\newtheorem{defn}{Definition}[section]
\newtheorem{theo}[defn]{Theorem}
\newtheorem{lem}[defn]{Lemma}
\newtheorem{prop}[defn]{Proposition}
\newtheorem{cor}[defn]{Corollary}
\newtheorem{rem}[defn]{Remark}
\newtheorem{exam}[defn]{Example}

\newenvironment{proof}{{\bf Proof }}{{\vskip 0.1cm \hfill$\Box$}}

\begin{document} 

\noindent
{\Large \bf Local elliptic regularity for solutions to stationary Fokker-Planck equations via Dirichlet forms and resolvents}
\\ \\
\bigskip
\noindent
{\bf Haesung Lee}  \\
\noindent
{\bf Abstract.}  
In this paper, we show that, for a solution to the stationary Fokker-Planck equation with general coefficients, defined as a measure with an $L^2$-density, this density not only exhibits $H^{1,2}$-regularity but also H\"{o}lder continuity.
To achieve this, we first construct a reference measure $\mu=\rho dx$ by utilizing existence and elliptic regularity results, ensuring that the given divergence-type operator corresponds to a sectorial Dirichlet form. By employing elliptic regularity results for homogeneous boundary value problems in both divergence and non-divergence type equations, we demonstrate that the image of the resolvent operator associated with the sectorial Dirichlet form has $H^{2,2}$-regularity. Furthermore, through calculations based on the Dirichlet form and the $H^{2,2}$-regularity of the resolvent operator, we prove that the density of the solution measure for the stationary Fokker-Planck equation is, indeed, the weak limit of $H^{1,2}$-functions defined via the resolvent operator. Our results highlight the central role of Dirichlet form theory and resolvent approximations in establishing the regularity of solutions to stationary Fokker-Planck equations with general coefficients.


\section{Introduction} 

The connection between stochastic differential equations (SDEs) and partial differential equations (PDEs) is often established through the Fokker-Planck equations. To explain this connection, consider a diffusion process  
\[
\mathbb{M}=(\Omega, \mathcal{F}, (\mathbb{P}_x)_{x \in \mathbb{R}^d}, (\mathcal{F}_t)_{t \geq 0}, (X_t)_{t \geq 0}),
\]  
satisfying the following It\^{o}-stochastic differential equation: for each $x \in \mathbb{R}^d$,
\[
X_t = x + \int_0^t \sigma(X_s) \, dW_s + \int_0^t \mathbf{G}(X_s) \, ds, \quad 0 \leq t < \infty, \quad \mathbb{P}_x\text{-a.s.},
\]  
where \( (W_t)_{t \geq 0} \) is a standard \( (\mathcal{F}_t)_{t \geq 0} \)-Brownian motion, \( \sigma = (\sigma_{ij})_{1 \leq i,j \leq d} \) is a \( d \times d \) matrix of functions, and \( \mathbf{G} \) is a vector field. Here, both \( \sigma \) and \( \mathbf{G} \) are assumed to satisfy certain regularity and growth conditions. If we additionally assume that the one-dimensional marginal distribution of \( (X_t)_{t > 0} \) under \( \mathbb{P}_x \) has a density \( (p^x_t(y))_{t > 0} \) with respect to \( dy \), i.e.
\[
\quad \mathbb{P}_x(X_t \in E) = \int_{E} p_t^x(y) \, dy, \quad \text{for any } E \in \mathcal{B}(\mathbb{R}^d),\, t>0, \,x \in \mathbb{R}^d,
\]
then it can be derived by It\^{o}'s formula and some calculations that, for any \( f \in C_0^{\infty}(\mathbb{R}^d) \) and \( g \in C_0^{\infty}((0, \infty)) \),
\begin{equation} \label{fokplanck}
\int_0^{\infty}\int_{\mathbb{R}^d} \Big(\partial_t (f(y)g(t)) + \frac12 \text{trace} (A \nabla^2 (f(y)g(t))) + \Big\langle \mathbf{G}, \nabla (f(y)g(t)) \Big\rangle \Big) p_t^x(y) \, dy \, dt = 0,
\end{equation}
where \( A =(a_{ij})_{1 \leq i,j \leq d}:= \sigma \sigma^T \). In that case, \( (p^x_t(y) \, dy \, dt)_{t > 0} \) is called a solution to the Fokker-Planck equation. Various analytic properties of solutions to Fokker-Planck equations with rough coefficients are systematically studied in \cite{BKR09, BKRS15, BRS23}. Furthermore, one can investigate the stationary counterpart to the above Fokker-Planck equation. For instance, if \( \nu \) is an invariant probability measure for \( \mathbb{M} \) (cf. \cite[Definition 3.43]{LST22}), i.e. \( \nu \) is a probability measure on \( \mathcal{B}(\mathbb{R}^d) \) satisfying that
\[
\int_{\mathbb{R}^d} \int_{\mathbb{R}^d} p_t^x(y) f(y) \, dy \, \nu(dx) = \int_{\mathbb{R}^d} f(x) \, \nu(dx), \quad \text{for all } f \in C_0^{\infty}(\mathbb{R}^d),
\]
then integrating \eqref{fokplanck} with respect to \( \nu(dx) \) and applying Fubini's theorem, we obtain that
\[
\int_{\mathbb{R}^d} Lf \, d\nu = 0, \quad \text{for all } f \in C_0^{\infty}(\mathbb{R}^d),
\]
where
\[
Lf := \frac12 \text{trace}(A \nabla^2 f)  + \langle \mathbf{G},  \nabla f \rangle, \quad f \in C_0^{\infty}(\mathbb{R}^d).
\]
In this regard, \( \nu \) is referred to as a solution to the stationary Fokker-Planck equation or an infinitesimally invariant measure (cf. \cite{LT22}), and formally it is written as \( L^* \nu = 0 \) (note that for the equation to make sense, it is required that $\|\mathbf{G}\|+|c| \in L^1_{loc}(\mathbb{R}^d, \nu)$).
Concerning this type of stationary Fokker-Planck equations, extensive studies have been conducted by \cite{BKR09, BRS12, BKRS15, BS17, B18, BRS23f}.\\
When appropriate regularity conditions are fulfilled for the diffusion process $\mathbb{M}$, the invariant probability measure $\nu$ can correspond to the limiting probability measure. In other words, the measure $\nu$ is represented as the limiting measure of the one-dimensional marginal distribution $p_t^x(y) \, dy$ as \( t \to \infty \) (see \cite[Section 3.2.3]{LST22}). This property enables invariant measures to provide a deterministic framework for analyzing random phenomena, making them highly applicable across various fields. Recently, they have been actively utilized in applications such as Markov Chain Monte Carlo (MCMC) (cf. \cite{HHS05,MCF15}) and image generation (cf. \cite{S19, SS21}). \\
If the existence and uniqueness of a solution $\nu$ to the stationary Fokker-Planck equation are ensured, one can further study precise numerical computations to approximate $\nu$. From a probabilistic perspective, a priori admissible regularity of $\nu$ is sufficiently described by its absolute continuity with respect to the Lebesgue measure, or more generally, by considering $\nu$ as a locally finite measure. However, such a level of low regularity is inadequate from a numerical standpoint. Specifically, when approximating solutions using Galerkin methods, it is natural to require the solution to exhibit at least local $H^{1,2}$-regularity (cf. \cite{BS08}). 
This requirement is not limited to classical numerical methods, but is also necessary for recent numerical methods, such as physics-informed neural networks (see \cite{YL24} and references therein). In recent approaches involving physics-informed neural networks, where trial functions are constructed using neural networks, the continuity of the true solution is expected to play an important role in ensuring the stability of error analysis, particularly from the perspective of the universal approximation theorem (cf. \cite{Cy89, HSW89}). Consequently, determining whether the density of $\nu$ with respect to the Lebesgue measure belongs to $H^{1,2}_{\text{loc}}(\mathbb{R}^d) \cap C(\mathbb{R}^d)$ emerges as a mathematically significant problem. \\
It is well known that if the coefficients of $L$ and $c$ are all smooth and $\text{det}(A)>0$, then any $\nu$ satisfying $(L+c)^* \nu = 0$ must admit a smooth density with respect to $dx$, as established in \cite[Chapter III]{Ta81}, which is considered as a generalization of Weyl's lemma (\cite[Corollary 2.2.1]{Jo13}).
On the other hand, when the coefficients of $L$ and $c$ are locally H\"{o}lder continuous and $\det(A) > 0$, any $\nu$ satisfying $(L+c)^* \nu = 0$ is guaranteed to have a locally H\"{o}lder continuous density with respect to $dx$, as verified in \cite{S73}. Recent results ensuring the continuity of the density in cases with more singular drifts and zero-order terms have been extensively studied by V.I. Bogachev and S.V. Shaposhnikov. Precisely, in \cite[Theorem 3.1]{BS17}, it was established that when the components of $A$ are locally H\"{o}lder continuous, $\det(A) > 0$, $\mathbf{G} \in L^p_{\text{loc}}(\mathbb{R}^d, \mathbb{R}^d)$, and $c \in L^p_{\text{loc}}(\mathbb{R}^d)$, the density of $\nu$ satisfying $(L + c)^* \nu = 0$ is locally H\"{o}lder continuous. Moreover, the same theorem demonstrates that the continuity of the density of $\nu$ remains robust even when the components of $A$ are relaxed to being Dini-continuous. \\
Meanwhile, various results regarding the local \( H^{1,2} \)-regularity of the density of \(\nu\) have been also presented. It has been shown in \cite{F77} that if the components of \(A\) are locally Lipschitz continuous, \( \text{det}(A) > 0 \), and \( c \in L^2_{\text{loc}}(\mathbb{R}^d) \), then any \(\nu=hdx\) with $h \in L^2_{loc}(\mathbb{R}^d)$ satisfying $\big(-\text{div}(A \nabla) + c\big)^* \nu = 0$ always possesses a density $h \in  H^{1,2}_{\text{loc}}(\mathbb{R}^d)$. Indeed, this local \( H^{1,2} \)-regularity result is motivated by the purpose of demonstrating the essential self-adjointness of the corresponding operator. From this perspective, \cite{BKR97} study the case where  $A =id$, establishing the general Sobolev regularity of the solution \( \nu \) and essential self-adjointness of the corresponding Dirichlet operators. Precisely, it is shown in \cite[Theorem 1]{BKR97} that if \( A = id \), \( \mathbf{G} \in L^p_{\text{loc}}(\mathbb{R}^d, \mathbb{R}^d) \), and \( c \in L^{\frac{pd}{p+d}}_{loc}(\mathbb{R}^d) \) with \( p \in (d, \infty) \), then \(\nu\) which is a solution to $(L + c)^* \nu = 0$ has a density in \( H^{1,p}_{\text{loc}}(\mathbb{R}^d) \cap C_{loc}^{0, 1-d/p}(\mathbb{R}^d) \) with respect to the Lebesgue measure. Subsequently, this result is generalized to the case where the components of \( A \) satisfy \( H^{1,p}_{\text{loc}}(\mathbb{R}^d) \) with \( p \in (d, \infty) \) and \( \text{det}(A) > 0 \) and detailed results on this generalization are provided in \cite{BKR01, BKR09, BKRS15}. As additional references, we refer to \cite{KO23, DKK} for recent results on the regularity of very weak solutions in the double divergence form, which is another name for the stationary Fokker-Planck equations.
\\
Meanwhile, a novel approach to show the $H^{1,2}$-regularity of the density of $\nu$ which is a solution to $L^* \nu=0$ was introduced by W. Stannat. Through the results of \cite[Theorem 2.1]{St99}(cf. \cite[Theorem 2.20]{LST22}), it is shown that if the components of $A$ are locally H\"{o}lder continuous, \(\mathbf{G}\) merely satisfies \(L^2_{\text{loc}}(\mathbb{R}^d, \mathbb{R}^d)\), and the density of \(\nu\), which satisfies \(L^* \nu = 0\), is assumed to be a priori in \(L^{\infty}(\mathbb{R}^d)\), then the density of \(\nu\) with respect to the Lebesgue measure is shown to belong to \(H^{1,2}_{\text{loc}}(\mathbb{R}^d)\). The main idea employed in this approach is to construct a resolvent operator $(V_{\alpha})_{\alpha > 0}$ corresponding to the non-divergence type operator $\sum_{i,j=1}^d a_{ij} \partial_{i} \partial_j$ and to demonstrate, via the Banach-Alaoglu theorem, that for $\nu = h \, dx$ with $L^* \nu=0$, the function $h$ is indeed a local $H^{1,2}$-weak limit of $(\alpha V_{\alpha} h)_{\alpha > 0}$ as $\alpha \rightarrow \infty$. This idea was subsequently inherited in \cite{L23ID} via the Dirichlet form theory. Indeed, it was shown in \cite[Theorem 3.3]{L23ID} that if $A = id$ and $\mathbf{G} = \frac{1}{\rho} \nabla \rho$ with $d \geq 3$, where $\rho$ is locally bounded below and above by strictly positive constants and $\nabla \rho \in L^d_{\mathrm{loc}}(\mathbb{R}^d, \mathbb{R}^d)$, then for any $\nu = h \, dx$ satisfying $L^* \nu = 0$ with $h \in L^2_{\mathrm{loc}}(\mathbb{R}^d)$, it follows from resolvent approximations that $h \in H^{1,2}_{\mathrm{loc}}(\mathbb{R}^d) \cap C(\mathbb{R}^d)$. Building upon the above results, we now introduce our main results as follows:
\begin{theo} \label{intromainth}
Let $U$ be a (possibly unbounded) open subset of $\mathbb{R}^d$ with $d \geq 3$. Let $\mathbf{H} \in L^p_{loc}(U, \mathbb{R}^d)$ with $p \in (d, \infty)$, and let $A = (a_{ij})_{1 \leq i,j \leq d}$ be a (possibly non-symmetric) matrix of functions, where $a_{ij} \in VMO_{loc}(U)$ for all $1 \leq i,j \leq d$ (see Definition \ref{defnvmo}), and ${\rm div} A \in L^d_{loc}(U, \mathbb{R}^d)$ (see Definition \ref{basidefn}). Assume that $A$ is locally uniformly strictly elliptic and bounded on $U$, i.e.  for each open ball $V$ with $\overline{V} \subset U$ there exist strictly positive constants $\lambda_V$ and $M_V$ such that
\begin{equation} \label{locellst}
\lambda_V \| \xi \|^2 \leq \langle A(x) \xi, \xi \rangle, \quad \max_{1 \leq i,j \leq d} |a_{ij}(x)| \leq M_V, \quad \text{ for a.e. $x \in V$ and for all $\xi \in \mathbb{R}^d$}. 
\end{equation}
As in Proposition \ref{convdivnon}, consider a partial differential operator $(\mathcal{L}, C_0^{\infty}(U))$ given by
\begin{align*}
\mathcal{L} f &= \text{\rm div} (A \nabla f) + \langle \mathbf{H}, \nabla f \rangle \nonumber \\
&= \text{\rm trace}(A \nabla^2  f) + \langle \text{\rm div} A + \mathbf{H}, \nabla f \rangle, \quad f \in C^{\infty}_0(U). 
\end{align*}
Let $c \in L^d_{loc}(U)$, $\mathbf{F} \in L^q_{loc}(U, \mathbb{R}^d)$ with $q \in [2, \infty)$ and $f \in L^{\frac{qd}{q+d}}_{loc}(U)$.
Assume that $h \in L^2_{loc}(U)$ satisfies
\begin{equation} \label{basicintideneq}
\int_{U} \big(\mathcal{L} \varphi +c\varphi \big) \, h dx = \int_{U} f \varphi dx   + \int_{U} \langle \mathbf{F}, \nabla \varphi \rangle dx , \quad \text{ for all $\varphi \in C_0^{\infty}(U)$}.
\end{equation}
Then, $h \in H^{1,p \wedge q}_{loc}(U)$. In particular, if $q \in (d, \infty)$, then $h \in C^{0, \beta}_{loc}(U)$ with $\beta=1-\frac{d}{p \wedge q}$ by the Sobolev embedding.
\end{theo}
\centerline{}
\noindent
The main idea for the proof of Theorem \ref{intromainth} is to find a suitable reference measure $\mu=\rho dx$ (see Theorem \ref{exisinfinvar}), so that
we can convert the divergence type operator $\mathcal{L}$ to $L^0+ \langle \mathbf{B}, \nabla \rangle$ (see \eqref{mathcallrep}), where $L^0$ is an operator associated to a sectorial Dirichlet form on $L^2(B, \mu)$ (see \eqref{secdirich}), $B$ is an open ball and $\mathbf{B}$ is a weakly divergence-free vector field with respect to $\mu$. As a result, due to the sufficient integrability of $\mathbf{B}$, the sectorial Dirichlet form $(\mathcal{E}^B, D(\mathcal{E}^B))$ corresponding to $L^0 + \langle \mathbf{B}, \nabla \rangle$ can be constructed as the closure of \eqref{predirici} in $L^2(B, \mu)$. Ultimately, by leveraging the corresponding resolvent through Dirichlet form theory and elliptic regularity results for homogeneous boundary value problems in Theorem \ref{resolregul}, we inherit the original idea of \cite[Theorem 3.3]{St99} and, through detailed computations, obtain the main result, Theorem \ref{intromainth}. Ultimately, by resolvent approximations, we demonstrate that \( h \) is indeed the $ H^{1,2}$-weak limit of a subsequence of $(\alpha G_{\alpha}^B h)_{\alpha > 0}$ as $\alpha \to 0^+$, where $(G^B_{\alpha})_{\alpha > 0}$ denotes the resolvent associated with the Dirichlet form $(\mathcal{E}^B, D(\mathcal{E}^B))$. As a direct consequence of Theorem \ref{intromainth}, we present the following local elliptic regularity result for non-divergence type operators:
\centerline{}
\begin{cor} \label{directcorol}
Let $U$ be a (possibly unbounded) open subset of $\mathbb{R}^d$ with $d \geq 3$. Let $\tilde{\mathbf{H}} \in L^p_{loc}(U, \mathbb{R}^d)$ with $p \in (d, \infty)$, $c \in L^d_{loc}(U)$, and let $A = (a_{ij})_{1 \leq i,j \leq d}$ be a (possibly non-symmetric) matrix of functions where $a_{ij} \in VMO_{loc}(U)$ for all $1 \leq i,j \leq d$ (see Definition \ref{defnvmo}), and ${\rm div} A \in L^p_{loc}(U, \mathbb{R}^d)$ (see Definition \ref{basidefn}). Assume that $A$ is locally uniformly strictly elliptic and bounded on $U$, i.e. \eqref{locellst} holds. Let $\mathbf{F} \in L^q_{loc}(U, \mathbb{R}^d)$ with $q \in [2, \infty)$, and $f \in L^{\frac{qd}{q+d}}_{loc}(U)$. Assume that $h \in L^2_{loc}(U)$ satisfies
$$
\int_{U} \Big( {\rm trace}(A \nabla^2 \varphi) + \langle \tilde{\mathbf{H}}, \nabla \varphi \rangle + c \varphi \Big) \, h \, dx = \int_{U} f \varphi \, dx + \int_{U} \langle \mathbf{F}, \nabla \varphi \rangle \, dx , \quad \text{for all $\varphi \in C_0^{\infty}(U)$}.
$$
Then, $h \in H^{1,p \wedge q}_{loc}(U)$. In particular, if $q \in (d, \infty)$, then $h \in C^{0, \beta}_{loc}(U)$ with $\beta=1-\frac{d}{p \wedge q}$ by the Sobolev embedding.
\end{cor}
\centerline{}
Combining Corollary \ref{directcorol} and \cite[Theorem 2.1]{BS17}, we directly obtain the following regularity result for general signed measure $\nu$:
\begin{cor} \label{finmaincor}
Let $U$ be a (possibly unbounded) open subset of $\mathbb{R}^d$ with $d \geq 3$. Let $\tilde{\mathbf{H}} \in L^p_{loc}(U, \mathbb{R}^d)$, $c \in L^p_{loc}(U)$ with $p \in (d, \infty)$, and let $A = (a_{ij})_{1 \leq i,j \leq d}$ be a (possibly non-symmetric) matrix of functions where $a_{ij} \in VMO_{loc}(U)$ for all $1 \leq i,j \leq d$ (see Definition \ref{defnvmo}), and ${\rm div} A \in L^p_{loc}(U, \mathbb{R}^d)$ (see Definition \ref{basidefn}). Assume that $A$ is locally uniformly strictly elliptic and bounded on $U$, i.e. \eqref{locellst} holds. Let $\nu$ be a locally finite signed measure satisfying that $\|\tilde{\mathbf{H}} \|  + |c| \in L^1_{loc}(U, \nu)$ and
$$
\int_{U} \Big( {\rm trace}(A \nabla^2 \varphi) + \langle \tilde{\mathbf{H}}, \nabla \varphi \rangle + c \varphi \Big) \,d\nu = 0, \quad \text{for all $\varphi \in C_0^{\infty}(U)$}.
$$
Then, there exists $h \in H^{1,p}_{loc}(U) \cap C^{0, 1-d/p}_{loc}(U)$ such that $\nu=h dx$.
\end{cor}
\centerline{}
This paper is structured as follows: Section \ref{sec2} provides a comprehensive explanation of the notations and conventions that are mainly used throughout this paper. Section \ref{2ndsection} establishes the critical framework for deriving the main results of this paper by constructing sectorial Dirichlet forms and their corresponding resolvents and generators. Section \ref{sec4} presents the detailed proofs of the main results introduced in the introduction. Section \ref{sec5} provides essential auxiliary results that are instrumental for the completeness and accuracy of the main proofs.  Finally, Section \ref{condisc} presents a brief conclusion of this paper and discusses directions for future research.

\section{Notations and conventions} \label{sec2}

In this study, we consider the Euclidean space $\mathbb{R}^d$ with $d \geq1$, which is equipped with the standard Euclidean inner product $\langle \cdot, \cdot \rangle$ and the corresponding Euclidean norm $\|\cdot\|$. For any point $x_0 \in \mathbb{R}^d$ and radius $r > 0$, the open ball centered at $x_0$ with radius $r$ is denoted by $B_r(x_0) := \{x \in \mathbb{R}^d : \|x - x_0\| < r\}$. For real numbers $a, b \in \mathbb{R}$ with $a \leq b$, we use the notation $a \wedge b=b \wedge a := a$ and $a \vee b =b \vee a:= b$. 
Let $dx$ denote the Lebesgue measure on $\mathbb{R}^d$. 
Let $U$ be an open subset of $\mathbb{R}^d$. The set of all Borel measurable functions on $U$ is denoted by $\mathcal{B}(U)$. For $\mathcal{A} \subset \mathcal{B}(U)$, $\mathcal{A}_0$ consists of functions $f \in \mathcal{A}$ such that $\text{supp}(f \, dx)$ is compact and contained in $U$. The space of continuous functions on $U$ and its closure $\overline{U}$ are denoted by $C(U)$ and $C(\overline{U})$, respectively.
We let $C_0(U) := C(U)_0$, which is the set of continuous functions on $U$ with compact support in $U$.
For $k \in \mathbb{N} \cup \{\infty\}$, the space of $k$-times continuously differentiable functions on $U$ is denoted by $C^k(U)$, and $C^k_0(U) := C^k(U) \cap C_0(U)$. For $k \in \mathbb{N} \cup \{ \infty \}$, the space $C^k(\overline{U})$ consists of all functions $f$ on $\overline{U}$ for which there exist an open set $V \supset \overline{U}$ and a function $\tilde{f} \in C^k(V)$ such that $\tilde{f} = f$ on $\overline{U}$. Let $r \in [1, \infty]$. The $L^r$-space on $U$ with respect to a measure $\nu$ is denoted by $L^r(U, \nu)$, equipped with the standard $L^r$-norm with respect to $\nu$. The space $L^r(U, \mathbb{R}^d, \nu)$ denotes the space of all $L^r$-vector fields on $U$ with respect to $\nu$ with the norm $\|\mathbf{F}\|_{L^r(U, \nu)} := \|\|\mathbf{F}\|\|_{L^r(U, \nu)}$. The localized $L^r$-space with respect to $\nu$, $L^r_{\text{loc}}(U, \nu)$, is the set of all Borel measurable functions $f$ on $U$ such that $f 1_W \in L^r(U, \nu)$ for any bounded open set $W \subset \mathbb{R}^d$ with $\overline{W} \subset U$. Similarly, $L^r_{\text{loc}}(U, \mathbb{R}^d, \nu)$ is the set of all vector fields $\mathbf{F}$ satisfying $\|\mathbf{F}\| \in L^r_{\text{loc}}(U, \nu)$. For simplicity, we write $L^r(U):=L^r(U, dx)$, $L_{loc}^r(U):= L^r_{loc}(U, dx)$, $L^r(U, \mathbb{R}^d):= L^r(U, \mathbb{R}^d, dx)$, and $L^r_{loc}(U, \mathbb{R}^d):=L^r_{loc}(U, \mathbb{R}^d, dx)$. For a function $f$ defined on $U$, the weak partial derivative with respect to the $i$-th coordinate is denoted by $\partial_i f$, provided it exists. The Sobolev space $H^{1,r}(U)$ consists of all functions $f \in L^r(U)$ for which $\partial_i f \in L^r(U)$ for each $i = 1, \ldots, d$, equipped with the usual $H^{1,r}(U)$-norm. The subspace $H^{1,q}_0(U)$, for $q \in [1, \infty)$, is defined as the closure of $C_0^\infty(U)$ in $H^{1,q}(U)$. The Sobolev space $H^{2,r}(U)$ consists of all functions $f \in L^r(U)$ with $\partial_i f \in L^r(U)$ and $\partial_i \partial_j f \in L^r(U)$ for all $i, j = 1, \ldots, d$ equipped with the usual $H^{2,r}(U)$-norm. The weak Laplacian is denoted by $\Delta f := \sum_{i=1}^d \partial_i \partial_i f$. For a twice weakly differentiable function $f$, the weak Hessian matrix is given by $\nabla^2 f := (\partial_i \partial_j f)_{1 \leq i, j \leq d}$. Let $\beta \in (0,1]$. Then, define $C^{0, \beta}(\overline{U})$ as the set of all continuous functions $f$ on $\overline{U}$ satisfying that
$\sup_{x,y \in \overline{U}} \frac{|f(x)-f(y)|}{\|x-y\|^{\beta}}< \infty$. $C^{0, \beta}_{loc}(U)$ denotes the set of all continuous functions $f$ on $U$ for which $f \in C^{0, \beta}(\overline{V})$ for any bounded open set $V$ with $\overline{V} \subset U$. For a matrix $B = (b_{ij})_{1 \leq i, j \leq d}$, its trace is defined as $\text{trace}(B) := \sum_{i=1}^d b_{ii}$. Throughout this work, we let $2_* := \frac{2d}{d+2}$ if $d \geq 3$.

\begin{defn} \label{defnvmo}
Let $d \geq 1$, $\omega$ be a positive continuous function on $[0, \infty)$ with $\omega(0) = 0$. The set $VMO$ consists of all functions $g \in L^1_{loc}(\mathbb{R}^d)$ for which
$$
\sup_{z \in \mathbb{R}^d, r < R} r^{-2d} \int_{B_r(z)} \int_{B_r(z)} |g(x) - g(y)| \, dx \, dy \leq \omega(R), \quad \text{for every $R > 0$}.
$$
For an open ball $B$ in $\mathbb{R}^d$, $VMO(B)$ denotes all functions $f \in L^1 (B)$ for which there exists $\tilde{f} \in VMO$ such that $\tilde{f}|_B=f$. For an open subset $U$ of $\mathbb{R}^d$, $VMO_{loc}(U)$ denotes all functions $f \in L^1_{loc}(U)$ for which $f \in VMO(B)$ for any open ball $B$ in $\mathbb{R}^d$ with $\overline{B} \subset U$.
\end{defn}
It is known that $C_0(\mathbb{R}^d) \cup W^{1,d}(\mathbb{R}^d) \subset VMO$ (see \cite[Proposition 2.1]{L23ID}). Using simple extension properties and the partition of unity, it follows that for any open subset $U$ of $\mathbb{R}^d$, $C(U) \cup W^{1,d}_{loc}(U) \subset VMO_{loc}(U)$. By the definition of $VMO$, if $f,g \in VMO$, then $f+g \in VMO$. Moreover, Proposition \ref{algprovmo} demonstrates that the set of bounded $VMO$ functions is closed under multiplication.

\begin{defn} \label{basidefn} 
(i) Let $U \subset \mathbb{R}^d$ be an open set with $d \geq 2$, $B = (b_{ij})_{1 \leq i,j \leq d}$ be a (possibly non-symmetric) matrix of locally integrable functions on $U$ and $\mathbf{E}=(e_1,\ldots, e_d) \in L^1_{loc}(U, \mathbb{R}^d)$. We write ${\rm{div}}B = \mathbf{E}$ if
\begin{equation} \label{defndivma}
\int_{U} \sum_{i,j=1}^d b_{ij} \partial_i \phi_j \, dx = - \int_{U} \sum_{j=1}^d e_j \phi_j \, dx \quad \text{ for all $\phi_j \in C_0^{\infty}(U)$, $j=1, \ldots, d$}.
\end{equation}
In other words, if we consider the family of column vectors
\[
\mathbf{b}_j=\begin{pmatrix} 
b_{1j} \\ \vdots \\  b_{dj}
\end{pmatrix},\quad 1\le j\le d,\  \text{ i.e.}\  B=(\mathbf{b}_1 | \ldots |\mathbf{b}_d),
\]
then \eqref{defndivma} is rewritten as
\begin{equation*}
\int_{U} \sum_{j=1}^d \langle \mathbf{b}_j, \nabla \phi_j \rangle dx =-\int_{U} \langle \mathbf{E}, \Phi \rangle dx, \text{ \;\;\; for all $\Phi=(\phi_1, \ldots, \phi_d) \in C_0^{\infty}(U)^d$.}
\end{equation*}
(ii) Let $U \subset \mathbb{R}^d$ be an open set with $d \geq 2$, $\mathbf{F} \in L^1_{loc}(U, \mathbb{R}^d)$ and $f \in L^1_{loc}(U)$. We write ${\rm div} \mathbf{F}=f$ if
$$
\int_{U} \langle \mathbf{F}, \nabla \varphi \rangle dx = -\int_{U} f \varphi dx, \quad \text{ for all } \varphi \in C_0^{\infty}(U).
$$
\end{defn}
\centerline{}

\section{Dirichlet forms approach} \label{2ndsection}
\begin{theo} \label{exisinfinvar}
Let $A = (a_{ij})_{1 \leq i,j \leq d}$ be a (possibly non-symmetric) matrix of measurable functions on $\mathbb{R}^d$ with $d \geq 2$, which is locally uniformly strictly elliptic and bounded on $\mathbb{R}^d$ (see \eqref{locellst}). Let $\mathbf{H} \in L^p_{loc}(\mathbb{R}^d, \mathbb{R}^d)$ with $p \in (d, \infty)$.
Then, there exists $\rho \in H^{1,2}_{loc}(\mathbb{R}^d) \cap C(\mathbb{R}^d)$ with $\rho(x)>0$ for all $x \in \mathbb{R}^d$ such that
\begin{equation} \label{infinvamea}
\int_{\mathbb{R}^d} \langle A^T \nabla \rho - \rho \mathbf{H}, \nabla \varphi \rangle dx = 0, \quad \text{ for all $\varphi \in C_0^{\infty}(\mathbb{R}^d)$}.
\end{equation}
Equivalently, for the partial differential operator $(\mathcal{L}, C_0^{\infty}(\mathbb{R}^d))$ defined by
\[
\mathcal{L}f ={\rm div}(A \nabla f) + \langle \mathbf{H}, \nabla f \rangle, \quad f \in C_0^{\infty}(\mathbb{R}^d),
\]
the measure $\rho\, dx$ is an infinitesimally invariant measure for $(\mathcal{L}, C_0^{\infty}(\mathbb{R}^d))$, i.e. it holds that
\[
\int_{\mathbb{R}^d} \mathcal{L}\varphi \cdot \rho \, dx = 0 \quad \text{for all } \varphi \in C_0^{\infty}(\mathbb{R}^d).
\]
Moreover, if $a_{ij} \in VMO_{loc}(\mathbb{R}^d)$ for all $1 \leq i,j \leq d$, then $\rho \in H^{1,p}_{loc}(\mathbb{R}^d) \cap C^{0, 1-d/p}(\mathbb{R}^d)$.
\end{theo}
\noindent
\begin{proof}
First, by \cite[Theorem 2.27(i)]{LST22} there exists $\rho \in H^{1,2}_{loc}(\mathbb{R}^d) \cap C(\mathbb{R}^d)$ with $\rho(x)>0$ for all $x \in \mathbb{R}^d$ such that \eqref{infinvamea} holds. Using integration by parts, the measure $\rho\,dx$ is an infinitesimally invariant measure for $(\mathcal{L}, C_0^{\infty}(\mathbb{R}^d))$. Now observe that
$$
\int_{\mathbb{R}^d} \langle A^T \nabla \rho, \nabla \varphi \rangle dx = \int_{\mathbb{R}^d}\langle \rho \mathbf{H}, \nabla \varphi \rangle dx, \quad \text{ for all $\varphi \in C_0^{\infty}(\mathbb{R}^d)$},
$$
and hence $\rho \in H^{1,p}_{loc}(\mathbb{R}^d) \cap C(\mathbb{R}^d)$ by \cite[Theorem 1.8.3]{BKRS15}. 
\end{proof}

\centerline{}
\noindent
{\it
{\bf (Hyp)}: $d \geq 3$, $\mathbf{H} \in L_{loc}^p(\mathbb{R}^d, \mathbb{R}^d)$ with $p \in (d, \infty)$, and $A=(a_{ij})_{1 \leq i,j \leq d}$ is a (possibly non-symmetric) matrix of functions on $\mathbb{R}^d$ such that $a_{ij} \in VMO$ for all $1 \leq i,j \leq d$ (see Definition \ref{defnvmo}) and for some constants $\lambda$, $M>0$,
\begin{equation*} 
\lambda \| \xi \|^2 \leq \langle A(x) \xi, \xi \rangle, \quad \max_{1 \leq i,j \leq d} |a_{ij}(x)| \leq M, \quad \text{ for a.e. $x \in \mathbb{R}^d$ and for all $\xi \in \mathbb{R}^d$}. 
\end{equation*}
}
\centerline{}
\noindent
Under the assumption that {\bf (Hyp)} holds, let $\rho$ be as in Theorem \ref{exisinfinvar}, and
\begin{align} \label{rhodivfreeb}
\mu:=\rho dx, \qquad  \mathbf{B}:= \mathbf{H} -\frac{1}{\rho} A^T \nabla \rho, \quad \text{ on }\; \mathbb{R}^d.
\end{align}
 Then, Theorem \ref{exisinfinvar} deduces that $\rho \mathbf{B} \in L^p_{loc}(\mathbb{R}^d, \mathbb{R}^d)$ satisfies 
\begin{equation} \label{divfreecon}
\int_{\mathbb{R}^d} \langle \mathbf{B}, \nabla \varphi \rangle d\mu =\int_{\mathbb{R}^d} \langle \rho \mathbf{B}, \nabla \varphi \rangle dx=\int_{\mathbb{R}^d}  \langle \rho \mathbf{H} -A^T \nabla \rho, \nabla \varphi \rangle \,dx =0, \quad \text{ for all $\varphi \in C_0^{\infty}(\mathbb{R}^d)$}.
\end{equation}
In particular,
\begin{equation} \label{divfreeconalg}
\int_{\mathbb{R}^d}  \langle \mathbf{B}, \nabla \varphi \rangle \varphi d\mu = \frac12 \int_{\mathbb{R}^d}  \langle \mathbf{B}, \nabla \varphi^2 \rangle  d\mu = 0,\quad \text{ for all $\varphi \in C_0^{\infty}(\mathbb{R}^d)$}.
\end{equation}
Let $B$ be an open ball in $\mathbb{R}^d$ and $f,g \in C_0^{\infty}(B)$.
Since $\rho$ is bounded below and above by strictly positive constants on $\overline{B}$, it follows from the Cauchy-Schwarz inequality and the Sobolev inequality that
\begin{align}
\left| \int_{B} \langle \mathbf{B}, \nabla f \rangle g d\mu \right| &\leq \|\rho \mathbf{B}\|_{L^{d}(B)} \| \nabla f \|_{L^2(B)} \|g\|_{L^{\frac{2d}{d-2}}(B)} \nonumber \\
&\leq \|\rho \mathbf{B}\|_{L^{d}(B)} \| \nabla f \|_{L^2(B)} \cdot \gamma \|\nabla g\|_{L^{2}(B)} \nonumber  \\
& \leq \gamma \frac{\max_{\overline{B}} \rho}{\min_{\overline{B}} \rho} \| \mathbf{B} \|_{L^d(B)} \left( \int_{B} \| \nabla f \|^2 d\mu \right)^{1/2} \left( \int_{B} \| \nabla g \|^2 d\mu \right)^{1/2}\nonumber \\
& \leq  \frac{\gamma}{\lambda} \frac{\max_{\overline{B}} \rho}{\min_{\overline{B}} \rho} \| \mathbf{B} \|_{L^d(B)}  \left( \int_{B}  \langle A \nabla f, \nabla f \rangle d\mu \right)^{1/2} \left( \int_{B}  \langle A \nabla g, \nabla g \rangle d\mu  \right)^{1/2}, \label{weksecb}
\end{align}
where $\gamma:=\frac{2(d-1)}{d-2}$. Now, define a positive bilinear form $(\mathcal{E}^B, C_0^{\infty}(B))$ given by
\begin{equation} \label{predirici}
\mathcal{E}^{B}(f,g) = \int_{B}  \langle A \nabla f, \nabla g \rangle d\mu - \int_{B} \langle \mathbf{B}, \nabla f \rangle g d\mu, \quad f,g \in C_0^{\infty}(B).
\end{equation}
Consequently, the following result demonstrates that the bilinear form $(\mathcal{E}^B, C_0^{\infty}(B))$ is closable on $L^2(B, \mu)$ and its closure on $L^2(B, \mu)$ is a Dirichlet form.

\begin{theo} \label{matebisdiri}
Under the assumption of {\bf (Hyp)}, let $\rho \in H^{1,p}_{loc}(\mathbb{R}^d) \cap C(\mathbb{R}^d)$ be as in Theorem \ref{exisinfinvar} and $\mu$ and $\mathbf{B}$ be defined as in \eqref{rhodivfreeb}.
Let $B$ be an open ball in $\mathbb{R}^d$ and define $(\mathcal{E}^B, C_0^{\infty}(B))$ as the positive bilinear form given by \eqref{predirici}. 
Then, the following holds:
\begin{itemize}
\item[(i)]
$(\mathcal{E}^B, C_0^{\infty}(B))$ satisfies the strong sector condition, i.e.
$$
\mathcal{E}^{B}(f,g) \leq c \, \mathcal{E}^{B}(f,f)^{1/2} \mathcal{E}^{B}(g,g)^{1/2}, \quad \text{ for all $f,g \in C_0^{\infty}(B)$},
$$
where $c>0$ is a constant independent of $f,g\in C_0^{\infty}(B)$.
\item[(ii)]
$(\mathcal{E}^B, C_0^{\infty}(B))$  is closable on $L^2(B, \mu)$ (see \cite[Chapter I, Definition 3.1]{MR92}), and hence its closure denoted by $(\mathcal{E}^B, D(\mathcal{E}^B))$ is a coercive closed form on $L^2(B, \mu)$ (see \cite[Chapter I, Definition 2.4, Remark 3.2]{MR92}). Indeed, $(\mathcal{E}^B, D(\mathcal{E}^B))$ is the smallest coercive closed form on $L^2(B, \mu)$ that extends $(\mathcal{E}^B, C_0^{\infty}(B))$.

\item[(iii)]
$D(\mathcal{E}^B)=H^{1,2}_0(B)$ and
$$
\mathcal{E}^B(f,g)= \int_{B} \langle A \nabla f, \nabla g \rangle d\mu- \int_{B} \langle \mathbf{B}, \nabla f \rangle g d\mu, \quad \text{ for all $f,g \in D(\mathcal{E}^B)$}.
$$
In particular, 
$$
\mathcal{E}^B(f,f)= \int_{B}  \langle A \nabla f, \nabla f \rangle d\mu, \quad \text{ for all $f \in D(\mathcal{E}^B)$}.
$$

\item[(iv)]
$(\mathcal{E}^B, D(\mathcal{E}^B))$ is a Dirichlet form (refer to \cite[Chapter I, Definition 4.5]{MR92}).
\end{itemize}
\end{theo}
\begin{proof}
(i) Let $f,g \in C_0^{\infty}(B)$ be given. Applying the Cauchy-Schwartz inequality, \eqref{divfreeconalg} and \eqref{weksecb}, we have
\begin{align*}
\left| \mathcal{E}^{B}(f,g)  \right| &\leq  \left| \int_{B}  \langle A \nabla f, \nabla g \rangle d\mu \right| + \left| \int_{B} \langle \mathbf{B}, \nabla f \rangle g d\mu \right| \\
&\leq \mathcal{E}^{B}(f,f)^{1/2}  \mathcal{E}^{B}(g,g)^{1/2}  +\frac{\gamma}{\lambda} \frac{\max_{\overline{B}} \rho}{\min_{\overline{B}} \rho} \| \mathbf{B} \|_{L^d(B)}  \mathcal{E}^{B}(f,f)^{1/2}  \mathcal{E}^{B}(g,g)^{1/2}  \\
& = \left(1+\frac{\gamma}{\lambda} \frac{\max_{\overline{B}} \rho}{\min_{\overline{B}} \rho} \| \mathbf{B} \|_{L^d(B)} \right) \mathcal{E}^{B}(f,f)^{1/2}  \mathcal{E}^{B}(g,g)^{1/2}.
\end{align*}
(ii) 
To show the closability of $(\mathcal{E}^B, C_0^{\infty}(B))$, let $(f_n)_{n \geq 1} \subset C_0^{\infty}(B)$ be such that $\lim_{n \rightarrow \infty} f_n =0$ in $L^2(B, \mu)$ and $\lim_{n,m \rightarrow \infty}\mathcal{E}^B(f_n-f_m, f_n-f_m)=0$.  Then, \eqref{divfreeconalg} implies that
\begin{align*}
\int_{B} \langle A \nabla (f_n-f_m), \nabla (f_n-f_m) \rangle d\mu &= \int_{B} \langle A \nabla (f_n-f_m), \nabla (f_n-f_m) \rangle d\mu-\frac12 \int_{B} \langle \mathbf{B}, \nabla( f_n-f_m)^2 \rangle  d\mu \\
&=\mathcal{E}^B(f_n-f_m, f_n-f_m)  \longrightarrow 0 \;\; \text{ as $n,m \rightarrow \infty$}.
\end{align*}
Thus, we have
$$
\|\nabla (f_n-f_m) \|^2_{L^2(B, \mu)} \leq \frac{1}{\lambda}\int_{B} \langle A \nabla (f_n-f_m), \nabla (f_n-f_m) \rangle d\mu \longrightarrow 0, \quad \text{ as $n,m \rightarrow \infty$}.
$$
Since $L^2(B, \mu)$ is complete, for each $i=1,2, \ldots d$ there exists $g_i \in L^2(B, \mu)$ such that
$$
\lim_{n \rightarrow \infty} \partial_i f_n=g_i, \quad \text{ in $L^2(B, \mu)$}.
$$
Now let $\varphi \in C_0^{\infty}(B)$ and $i \in \{1, 2 \ldots, d\}$. Then,
$$
\int_{B} g_i \varphi dx = \lim_{n \rightarrow \infty} \int_{B} \partial_i f_n \cdot \varphi dx = \lim_{n \rightarrow \infty}-\int_{B} f_n \cdot \partial_i \varphi dx =0.
$$
Thus, $g_i=0$ on $B$, and hence $\lim_{n \rightarrow \infty} \nabla f_n=0$ in $L^2(B, \mathbb{R}^d)$, so that $\lim_{n \rightarrow \infty}\mathcal{E}^B(f_n, f_n)=0$. Therefore, $(\mathcal{E}^B, C_0^{\infty}(B))$ is closable. The rest follows from the weak sector condition in (i). \\ \\
(iii)
Let $f \in D(\mathcal{E}^B)$. Then, there exists $(f_n)_{n \geq 1} \subset C_0^{\infty}(B)$ such that $\lim_{n \rightarrow \infty}f_n=f$ in $L^2(B, \mu)$ and $\mathcal{E}^B(f_n-f_m, f_n-f_m) \longrightarrow 0$ as $n,m \rightarrow \infty$. Then, $\| f_n-f_m \|_{H^{1,2}_0(U)} \longrightarrow 0$ as $n,m \rightarrow \infty$. By the completeness of $H^{1,2}_0(B)$, there exists $\tilde{f} \in H^{1,2}_0(B)$ such that $\lim_{n \rightarrow \infty} f_n = \tilde{f}$ in $H^{1,2}_0(B)$. Therefore, we have $f=\tilde{f} \in H^{1,2}_0(B)$. Conversely, let $f \in H^{1,2}_0(B)$. Then, there exists $(f_n)_{n \geq 1} \subset C_0^{\infty}(B)$ such that $\lim_{n \rightarrow \infty} f_n = f$ in $H^{1,2}_0(B)$. Thus, $\lim_{n \rightarrow \infty} f_n =f$ in $L^2(B, \mu)$ and $\mathcal{E}^B(f_n-f_m, f_n-f_m) \longrightarrow 0$ as $n,m \rightarrow \infty$. Therefore, $f \in D(\mathcal{E}^B)$. For the rest assertion, let $f,g \in D(\mathcal{E}^B)$. Then, there exists $(f_n)_{n \geq 1}, (g_n)_{n \geq 1} \subset C_0^{\infty}(B)$ such that $\lim_{n \rightarrow \infty}f_n=f$, $\lim_{n \rightarrow \infty}g_n=g$ in $L^2(B, \mu)$ and $\mathcal{E}^B(f_n-f_m, f_n-f_m) \longrightarrow 0$, $\mathcal{E}^B(g_n-g_m, g_n-g_m) \longrightarrow 0$ as $n,m \rightarrow \infty$. Thus, based on the arguments above, $f,g \in H^{1,2}_0(B)$ satisfy $\lim_{n \rightarrow \infty} f_n =f$ and $\lim_{n \rightarrow \infty} g_n =g$ in $H^{1,2}_0(B)$. Finally, since $\mathcal{E}^B(f,g) = \lim_{n \rightarrow \infty} \mathcal{E}^B(f_n, g_n)$, the assertion follows.\\ \\
(iv)
Let $\varepsilon>0$ and $\varphi_{\varepsilon} \in C^{\infty}(\mathbb{R})$ be a function defined as in Proposition \ref{varphidefn}. First consider $u \in C_0^{\infty}(B)$. Then, $\varphi_{\varepsilon} \circ u \in C_0^{\infty}(B)$. Define a bilinear form $(\mathcal{E}^{0,B}, C_0^{\infty}(B))$ given by
\begin{equation} \label{secdirich}
\mathcal{E}^{0,B}(f,g) := \int_{B} \langle A \nabla f, \nabla g \rangle d\mu, \quad f,g \in C_0^{\infty}(B).
\end{equation}
Observe that
\begin{align*}
\mathcal{E}^{0,B}(\varphi_{\varepsilon} \circ u, u-\varphi_{\varepsilon} \circ u )  = \int_{B} (\varphi'_{\varepsilon} \circ u ) (1-\varphi'_{\varepsilon} \circ u ) \langle A \nabla u, \nabla u \rangle d\mu  \geq 0,
\end{align*}
and hence $\liminf_{ \varepsilon \rightarrow 0+} \mathcal{E}^{0,B}(\varphi_{\varepsilon} \circ u, u-\varphi_{\varepsilon} \circ u ) \geq 0$. Likewise, $\liminf_{ \varepsilon \rightarrow 0+} \mathcal{E}^{0,B}( u-\varphi_{\varepsilon} \circ u, \varphi_{\varepsilon} \circ u ) \geq 0$. Meanwhile, let
$$
S_{\varepsilon}:= -\int_{B} \big \langle \mathbf{B}, \nabla (\varphi_{\varepsilon} \circ u) \big \rangle (u-\varphi_{\varepsilon} \circ u)\,d\mu. 
$$
Then, by Lebesgue's theorem and Proposition \ref{auxpropfir}, we obtain that
\begin{align*} 
\lim_{\varepsilon \rightarrow 0+} S_{\varepsilon} &= \lim_{\varepsilon \rightarrow 0+} \int_{B} (\varphi'_{\varepsilon} \circ u ) (u-\varphi_{\varepsilon} \circ u) \langle -\mathbf{B}, \nabla u \rangle d\mu, \nonumber \\
&=\int_{B} (1_{[0,1]} \circ u)  (u-u^+ \wedge 1) \langle -\mathbf{B}, \nabla u \rangle d\mu = 0. 
\end{align*}
Likewise, let
$$
T_{\varepsilon}:=-\int_{B} \big \langle \mathbf{B}, \nabla (u-\varphi_{\varepsilon} \circ u) \big \rangle (\varphi_{\varepsilon} \circ u)\,d\mu.
$$
Then, by Lebesgue's theorem and Proposition \ref{auxpropfir}, we deduce that
\begin{align*}
\lim_{\varepsilon \rightarrow 0+} T_{\varepsilon} &= \lim_{\varepsilon \rightarrow 0+} \int_{B} (1-\varphi'_{\varepsilon} \circ u) (\varphi_{\varepsilon} \circ u) \langle -\mathbf{B}, \nabla u \rangle d\mu \\
&=\int_{B} (1_{\mathbb{R}\setminus[0,1]} \circ u)  (u^+ \wedge 1) \langle -\mathbf{B}, \nabla u \rangle d\mu = \int_{B} (1_{(1, \infty)} \circ u) \langle -\mathbf{B}, \nabla u \rangle d\mu.
\end{align*}
Now, let $\Phi_{\varepsilon} \in C^{\infty}(\mathbb{R})$ be as in Proposition \ref{auxprophi}. Then, $\Phi_{\varepsilon} \circ u \in C_0^{\infty}(B)$. Moreover, by Lebesgue's theorem and Proposition \ref{auxprophi},
\begin{align*}
\lim_{\varepsilon \rightarrow 0+} T_{\varepsilon}&=\int_{B} (1_{(1, \infty)} \circ u) \langle -\mathbf{B}, \nabla u \rangle d\mu = \lim_{\varepsilon \rightarrow 0+} \int_{B} (\Phi'_{\varepsilon} \circ u) \langle -\mathbf{B}, \nabla u \rangle d\mu \nonumber \\
&=\lim_{\varepsilon \rightarrow 0+} \int_{B} \langle -\mathbf{B}, \nabla (\Phi_{\varepsilon} \circ u)  \rangle d\mu \geq 0. 
\end{align*}
Therefore, 
\begin{align*}
&\liminf_{ \varepsilon \rightarrow 0+} \mathcal{E}^{B}(\varphi_{\varepsilon} \circ u, u-\varphi_{\varepsilon} \circ u ) =\liminf_{ \varepsilon \rightarrow 0+} \Big( \mathcal{E}^{0, B}(\varphi_{\varepsilon} \circ u, u-\varphi_{\varepsilon} \circ u ) + S_{\varepsilon} \Big) \\
& \geq  \liminf_{ \varepsilon \rightarrow 0+} \mathcal{E}^{0, B}(\varphi_{\varepsilon} \circ u, u-\varphi_{\varepsilon} \circ u )+\liminf_{ \varepsilon \rightarrow 0+} S_{\varepsilon} \geq 0
\end{align*}
and
\begin{align*}
&\liminf_{ \varepsilon \rightarrow 0+} \mathcal{E}^{B}(u-\varphi_{\varepsilon} \circ u, \varphi_{\varepsilon} \circ u ) =   \liminf_{ \varepsilon \rightarrow 0+} \Big(\mathcal{E}^{B}(u-\varphi_{\varepsilon} \circ u, \varphi_{\varepsilon} \circ u )+T_{\varepsilon} \Big) \\
&\geq  \liminf_{ \varepsilon \rightarrow 0+} \mathcal{E}^{0, B}(u-\varphi_{\varepsilon} \circ u, \varphi_{\varepsilon} \circ u )+\liminf_{ \varepsilon \rightarrow 0+} T_{\varepsilon} \geq 0.
\end{align*}
Since $C_0^{\infty}(B)$ is dense in $D(\mathcal{E}^B)$ with respect to $\mathcal{E}^{B}(\cdot, \cdot)^{1/2} + \| \cdot \|_{L^2(B, \mu)}$, we conclude that $(\mathcal{E}^B, D(\mathcal{E}^B))$ is a Dirichlet form by \cite[Chapter I, Proposition 4.10]{MR92}.
\end{proof}

\begin{rem}
Upon careful examination of the proof, it can be shown that even if condition \eqref{divfreecon} is replaced with the condition, 
$$
\int_{B} \langle \mathbf{B}, \nabla \varphi \rangle d\mu \leq 0, \quad \text{ for all $\varphi \in C_0^{\infty}(B)$ with $\varphi \geq 0$},
$$
$(\mathcal{E}^B, D(\mathcal{E}^B))$ still remains a Dirichlet form.
\end{rem}

\centerline{}


\noindent
From now on, under the assumption of {\bf (Hyp)} with an open ball $B$ in $\mathbb{R}^d$, additionally assume that $\text{div} A \in L^d(B, \mathbb{R}^d)$. From Proposition \ref{convdivnon}, we can define an operator $\mathcal{L}^B$ on $H^{2,2}(B)$ as follows:
\begin{align}
\mathcal{L}^B f &= \text{div} (A \nabla f) + \langle \mathbf{H}, \nabla f \rangle \nonumber \\
&= \text{trace}(A \nabla^2  f) + \langle \text{div} A + \mathbf{H}, \nabla f \rangle  \nonumber  \\
&= \text{trace}(A \nabla^2  f) + \Big \langle \text{div} A +\frac{1}{\rho} A^T\nabla \rho , \nabla f \Big \rangle  + \left \langle \mathbf{B}, \nabla f \right \rangle, \quad f \in H^{2,2}(B), \label{mathcallrep}
\end{align}
where $\rho$ and $\mathbf{B}$ are defined as in Theorem \ref{exisinfinvar} and 
\eqref{rhodivfreeb}, respectively. Then, $\mathcal{L}^Bf \in L^2(B, \mu)$ for any $ f \in H^{2,2}(B)$.
\begin{theo} \label{resolregul}
Under the assumption {\bf (Hyp)}, let $\rho \in H^{1,p}_{loc}(\mathbb{R}^d) \cap C(\mathbb{R}^d)$ be as in Theorem \ref{matebisdiri}, $\mu$ and $\mathbf{B}$ be defined as in \eqref{rhodivfreeb}. Let $B$ be an open ball in $\mathbb{R}^d$, $(\mathcal{E}^B, D(\mathcal{E}^B))$ be a Dirichlet form defined as in Theorem \ref{matebisdiri} and $(G^B_{\alpha})_{\alpha>0}$ be its corresponding strongly continuous sub-Markovian resolvent of contractions on $L^2(B, \mu)$ (cf. \cite[Chapter I, Theorem 2.8]{MR92}). Suppose that $\text{\rm div} A \in L^d(B, \mathbb{R}^d)$ and let $\alpha>0$  and $f \in L^{2}(B, \mu)$. Then $G_{\alpha}^B f \in H^{2,2}(B) \cap H^{1,\frac{2d}{d-2}}_0(B)$.
\end{theo}
\begin{proof}
Let $\alpha>0$  and $f \in L^{2}(B, \mu)$. By \cite[Chapter I, Theorem 2.8]{MR92}, $G_{\alpha} f \in D(\mathcal{E}^B)$ and
$$
\mathcal{E}^B (G^B_{\alpha} f, \varphi) + \alpha \int_{B} G_{\alpha} f \cdot \varphi d\mu =\int_{B} f \varphi d\mu, \quad \text{ for all  $\varphi \in D(\mathcal{E}^B)$}.
$$
According to Theorem \ref{matebisdiri}(iii), $G^B_{\alpha} f \in D(\mathcal{E}^B) = H^{1,2}_0(B)$ and
\begin{align}
&\int_{B} \langle \rho A \nabla G_{\alpha}f, \nabla \varphi \rangle dx- \int_{B} \langle \rho \mathbf{B}, \nabla G_{\alpha} f \rangle \varphi dx+\int_{B} \alpha \rho G_{\alpha} f \cdot \varphi dx \nonumber  \\
&=\int_{B} \langle A \nabla G_{\alpha}f, \nabla \varphi \rangle d\mu- \int_{B} \langle \mathbf{B}, \nabla G_{\alpha} f \rangle g d\mu+\alpha \int_{B} G_{\alpha} f \cdot \varphi d\mu \nonumber \\
&=\mathcal{E}^B (G^B_{\alpha} f, \varphi) + \alpha \int_{B} G_{\alpha} f \cdot \varphi d\mu  = \int_{B} f \varphi d\mu = \int_{B} \rho f \varphi dx, \quad \text{ for all $\varphi \in C_0^{\infty}(B)$}. \label{mainelipeq}
\end{align}
Let $x_0 \in \mathbb{R}^d$ and $r>0$ be such that $B=B_r(x_0)$. Let $\rho_e \in H^{1,p}(\mathbb{R}^d)_0$ be an extension of $\rho \in H^{1,p}(B_{2r}(x_0))$ (cf. \cite[Theorem 4.7]{EG15}) satisfying that $\rho_e=\rho$ in $\overline{B}_{2r}(x_0)$ and $\text{supp}(\rho_e) \subset B_{4r}(x_0)$. Take $\chi \in C_0^{\infty}(\mathbb{R}^d)$ such that $\chi(x)=1$ for all $x \in \overline{B}_{r}(x_0)$, $\text{supp}(\chi) \subset B_{2r}(x_0)$ and $0 \leq \chi(x) \leq 1$ for all $x \in \mathbb{R}^d$. Define a function $\tilde{\rho}: \mathbb{R}^d \rightarrow \mathbb{R}$ given by
$$
\tilde{\rho}(x):= \chi \Big( \rho_e(x) - \min_{\overline{B}_{2r}(x_0)} \rho \Big) + \min_{\overline{B}_{2r}(x_0)} \rho, \quad x \in \mathbb{R}^d.
$$
Then, by Proposition \ref{algprovmo}, $\tilde{\rho} \in H^{1,p}(\mathbb{R}^d) \cap VMO \cap L^{\infty}(\mathbb{R}^d)$ and it satisfies $\tilde{\rho}=\rho$ in $\overline{B}=\overline{B}_r(x_0)$ and
$$
\min_{\overline{B}_{2r}(x_0)} \rho \leq \tilde{\rho}(x) \leq  \max_{\overline{B}_{2r}(x_0)} \rho, \quad \text{ for all $x \in \mathbb{R}^d$}.
$$
Thus, we can rewrite \eqref{mainelipeq} as
\begin{equation} \label{ourbasiceq}
\int_{B} \langle \tilde{\rho} A \nabla G^B_{\alpha}f, \nabla \varphi \rangle dx- \int_{B} \langle \rho \mathbf{B}, \nabla G^B_{\alpha} f \rangle \varphi dx+\int_{B} \alpha \rho G^B_{\alpha} f \cdot \varphi dx = \int_{B} (\rho f) \varphi dx, \quad \text{ for all $\varphi \in C_0^{\infty}(B)$}.
\end{equation}
Note that by Proposition \ref{algprovmo}, $\tilde{\rho} a_{ij} \in VMO \cap L^{\infty}(\mathbb{R}^d)$ for all $1 \leq i,j \leq d$. Moreover,
$$
\lambda \Big( \min_{\overline{B}_{2r}(x_0)} \rho \Big)  \,\| \xi \|^2  \leq \langle \tilde{\rho} A(x) \xi, \xi \rangle, \quad \text{ for a.e. $x \in \mathbb{R}^d$ and all $\xi \in \mathbb{R}^d$},
$$
$\rho \mathbf{B} \in L^p(B, \mathbb{R}^d)$, and $-\alpha \rho G_{\alpha}^B f +\rho f \in L^{2}(B)$.
By using \cite[Theorem 2.1(i)]{KK19}, there exists $u \in H^{1,\frac{2d}{d-2}}_0(U)$ such that
\begin{equation} \label{ourbasiceqtran}
\int_{B} \langle \tilde{\rho} A \nabla u, \nabla \varphi \rangle dx- \int_{B} \langle \rho \mathbf{B}, \nabla u \rangle \varphi dx = \int_{B} (-\alpha \rho G_{\alpha}^B f +\rho f ) \varphi dx, \quad \text{ for all $\varphi \in C_0^{\infty}(B)$}.
\end{equation}
Applying the weak maximum principle in \cite[Corollary 3.3]{T73} (cf. \cite[Theorem 1]{T77}) to \eqref{ourbasiceq} and \eqref{ourbasiceqtran}, we obtain that $G^B_{\alpha}f=u \in H^{1,\frac{2d}{d-2}}_0(B)$. Now, let
$$
g:= \rho f  + \langle A^T \nabla \rho +\rho\, \text{div} A, \nabla G^B_{\alpha} f \rangle + \langle \rho \mathbf{B}, \nabla G^B_{\alpha} f \rangle - \alpha \rho G^B_{\alpha} f.
$$
Then, $g \in L^2(B)$. Using \cite[Theorem 4.4]{CFL93}, there exists $v \in H^{2,2}(B) \cap H^{1,\frac{2d}{d-2}}_0(B)$ such that
$$
-\text{trace}( \tilde{\rho} A \nabla^2 v) = g, \quad \text{ on } \;B,
$$ 
and hence using integration by parts in Proposition \ref{intbypartform}, we obtain that
$$
\int_{B} \langle \rho A \nabla v, \nabla \varphi \rangle dx + \int_{B} \langle  A^T \nabla \rho +\rho\, \text{div} A, \nabla v\rangle \varphi dx = \int_{B}  g  \varphi  dx, \quad \text{ for all $\varphi \in C_0^{\infty}(B)$}.
$$
Meanwhile, from \eqref{ourbasiceq}
$$
\int_{B} \langle \rho A \nabla G^B_{\alpha}f, \nabla \varphi \rangle dx + \int_{B} \langle  A^T \nabla \rho +\rho\, \text{div} A, \nabla   G^B_{\alpha}f \rangle \varphi dx = \int_{B}  g  \varphi  dx, \quad \text{ for all $\varphi \in C_0^{\infty}(B)$}.
$$
By the consequence of the weak maximum principle in \cite[Corollary 3.3]{T73} (cf. \cite[Theorem 1]{T77}), we finally get $G^B_{\alpha} g =v \in H^{2,2}(B) \cap H^{1,\frac{2d}{d-2}}_0(B)$, as desired. 
\end{proof}

\begin{prop} \label{generprop}
Given the assumption {\bf (Hyp)}, let $\rho \in H^{1,p}_{loc}(\mathbb{R}^d) \cap C(\mathbb{R}^d)$ be as in Theorem \ref{matebisdiri}, $\mu$ and $\mathbf{B}$ be defined as in \eqref{rhodivfreeb}. Let $B$ be an open ball in $\mathbb{R}^d$, $(\mathcal{E}^B, D(\mathcal{E}^B))$ be a Dirichlet form defined as in Theorem \ref{matebisdiri} and  $(L^B, D(L^B))$ be its corresponding generator on $L^2(B, \mu)$(cf. \cite[Chapter I, Corollary 2.10, Proposition 2.16]{MR92}). 
Suppose that $\text{\rm div} A \in L_{loc}^d(\mathbb{R}^d, \mathbb{R}^d)$, and let $u \in H^{2,2}(B) \cap H_0^{1,\frac{2d}{d-2}}(B)$ and $\mathcal{L}^B$ be defined as in \eqref{mathcallrep}.
Then, $u \in D(L^B)$ and $L^B u = \mathcal{L}^B u$, i.e.
\begin{align*}
L^Bu&=\text{\rm div} (A \nabla u) + \langle \mathbf{H}, \nabla u \rangle  \\
&=  \text{\rm trace}(A \nabla^2  u) + \langle \text{\rm div} A + \mathbf{H}, \nabla u \rangle  \\
&=  \text{\rm trace}(A \nabla^2  u) + \Big \langle \text{\rm div} A +\frac{1}{\rho} A^T\nabla \rho , \nabla u \Big \rangle  + \left \langle \mathbf{B}, \nabla u \right \rangle.
\end{align*}
Specifically,
\begin{align} \label{genproperty}
\mathcal{E}^B(u,v) = -\int_{B} L^B u \cdot v d\mu =-\int_{B} \mathcal{L}^B u \cdot v d\mu, \quad \text{ for all $v \in D(\mathcal{E}^B)$}.
\end{align}
\end{prop}
\begin{proof}
Let $u \in H^{2,2}(B) \cap H_0^{1,\frac{2d}{d-2}}(B)$. Then, $u \in D(\mathcal{E}^B) =H^{1,2}_0(B)$ by Theorem \ref{matebisdiri}(iii).
Applying Theorem \ref{matebisdiri}(iii) and Proposition \ref{intbypartform}, we find that
\begin{align}
\mathcal{E}^B(u,\varphi)&=\int_{B} \langle A \nabla u, \nabla \varphi \rangle d\mu - \int_{B} \langle \mathbf{B}, \nabla u \rangle \varphi d\mu \nonumber \\
&=-\int_{B} \Big( \rho\, \text{trace}(A \nabla^2  u) + \langle \rho \,\text{div} A +A^T\nabla \rho , \nabla u \rangle \Big) \varphi \,dx  - \int_{B} \langle \mathbf{H} - \frac{1}{\rho} A^T \nabla \rho, \nabla u \rangle \varphi d\mu  \nonumber \\
&= -\int_{B}\Big( \text{trace}(A \nabla^2  u) + \langle \text{div} A + \mathbf{H}, \nabla u \rangle \Big) \varphi d\mu, \quad \text{ for all $\varphi \in C_0^{\infty}(B)$}. \label{intbypartgener}
\end{align}
Consequently, \eqref{intbypartgener} extends to all $\varphi \in D(\mathcal{E}^B)$ through approximation. Since 
$$
\text{trace}(A \nabla^2  u) + \langle \text{div} A + \mathbf{H}, \nabla u \rangle \in L^2(B, \mu), 
$$
we conclude from \cite[Chapter I, Proposition 2.16]{MR92} and Proposition \ref{convdivnon} that the assertion is validated.
\end{proof}

\begin{prop} \label{algproper}
Under the assumption {\bf (Hyp)}, suppose that $\text{\rm div} A \in L^d(B, \mathbb{R}^d)$.  Let $\chi \in C_0^{\infty}(B)$ and $u \in H^{2,2}(B)\cap H_0^{1,\frac{2d}{d-2}}(B) $. Then, $\chi, u, \chi u \in H^{2,2}(B) \cap H_0^{1,\frac{2d}{d-2}}(B) \subset D(L^B)$ (by Proposition \ref{generprop}) and
$$
L^B (\chi u)=u L^B \chi + \chi L^B u +\langle A \nabla \chi, \nabla u \rangle + \langle A \nabla u, \nabla \chi \rangle.
$$
\end{prop}
\begin{proof}
First by the product rule and Sobolev's embedding, $\chi u \in H^{2,2}(U) \cap H_0^{1,\frac{2d}{d-2}}(B)$. Thus, by Proposition \ref{generprop}, 
$\chi, u, \chi u \in H^{2,2}(B) \cap H_0^{1,\frac{2d}{d-2}}(B) \subset D(L^B)$. Meanwhile,
using Propositions \ref{convdivnon} and \ref{propdruled},
$$
{\rm div} \Big(\chi A \nabla u \Big) =\chi \,{\rm div} \Big(A \nabla u \Big)  + \langle \nabla \chi, A \nabla u \rangle  \in L^2(B)
$$
and
$$
{\rm div} \Big(u A \nabla \chi \Big) =u\, {\rm div} \Big(A \nabla \chi \Big)  + \langle \nabla u, A \nabla \chi \rangle  \in L^2(B).
$$
Therefore, we obtain that 
\begin{align*}
L^B (\chi u) &= {\rm div} \Big(A \nabla (\chi u)\Big) + \langle \mathbf{H}, \nabla (\chi u) \rangle \\
&={\rm div} \Big(\chi A \nabla u \Big) + {\rm div} \Big(u A \nabla \chi \Big)+
\chi \langle \mathbf{H}, \nabla  u \rangle +u \langle \mathbf{H}, \nabla  \chi \rangle  \\
&= u \Big( \text{\rm div} (A \nabla \chi) + \langle \mathbf{H}, \nabla \chi \rangle \Big)+
 \chi \Big( \text{\rm div} (A \nabla u) + \langle \mathbf{H}, \nabla u \rangle \Big)+\langle A \nabla \chi, \nabla u \rangle + \langle A \nabla u, \nabla \chi \rangle \\
&=u L^B \chi + \chi L^B u +\langle A \nabla \chi, \nabla u \rangle + \langle A \nabla u, \nabla \chi \rangle,
\end{align*}
as desired.
\end{proof}

\section{Main results} \label{sec4}
\begin{theo} \label{mainresreg}
Assume that {\bf (Hyp)} holds and let $x_0 \in \mathbb{R}^d$ and $R>0$. Suppose that ${\rm div} A \in L_{loc}^d(B_R(x_0), \mathbb{R}^d)$ and let $\tilde{h} \in L^2_{loc}(B_R(x_0))$, $c \in L^{d}_{loc}(B_R(x_0))$, $\tilde{f} \in L^{\frac{2d}{d+2}}_{loc}(B_R(x_0))$ and $\tilde{\mathbf{F}} \in L^2_{loc}(B_R(x_0), \mathbb{R}^d)$ satisfy
$$
\int_{B_R(x_0)} \Big( \text{\rm div} \big( A \nabla u \big) + \langle \mathbf{H}, \nabla u \rangle +cu \Big) \tilde{h} dx = \int_{B_R(x_0)} \tilde{f} udx +\int_{B_R(x_0)} \langle \tilde{\mathbf{F}}, \nabla u \rangle dx \quad \text{ for all $u \in C_0^{\infty}(B_R(x_0))$}.
$$
Then, $\tilde{h} \in H^{1,2}_{loc}(B_R(x_0))$.
\end{theo}
\begin{proof}
Consider an arbitrary $s \in (0, R)$. Choose $r \in (s, R)$ and write $B:=B_r(x_0)$. Let $\chi \in C_0^{\infty}(B)$ with $\chi \equiv 1$ on $\overline{B}_{s}(x_0)$ and $0 \leq \chi \leq 1$ on $\mathbb{R}^d$.
To establish the assertion, it is enough to show that $\chi h \in D(\mathcal{E}^B)=H_0^{1,2}(B)$. Let $v \in H^{2,2}(B) \cap H^{1,2}_0(B)$ with $\text{supp}(v) \subset B$ be arbitrarily given. Utilizing mollification and the Sobolev embedding theorem, we can find a sequence $(v_n)_{n \geq 1} \subset C_0^{\infty}(B)$ such that
\begin{equation} \label{ourapprox}
\lim_{n \rightarrow \infty} v_n =v \; \text{ in $H^{2,2}(B)$} \quad \;  \lim_{n \rightarrow \infty} \nabla v_n =\nabla v, \;\; \text{ in $L^{\frac{2d}{d-2}}(B, \mathbb{R}^d)$}, \quad \text{ and } \;  \lim_{n \rightarrow \infty} v_n = v \; \text{ in $L^{\frac{2d}{d-2}}(B)$}.
\end{equation}
Observe that, according to Proposition \ref{convdivnon},
\begin{align*}
\text{div} (A \nabla v) + \langle \mathbf{H}, \nabla v \rangle = \text{trace}(A \nabla^2  v) + \langle \text{div} A + \mathbf{H}, \nabla v \rangle  \in L^2(B). \quad 
\end{align*}
Accordingly, through approximation in \eqref{ourapprox}, we deduce that
\begin{align} \label{intidenonb}
\int_{B} \Big( \text{\rm div} \big( A \nabla v \big) + \langle \mathbf{H}, \nabla v \rangle +cv \Big) \tilde{h} dx = \int_{B} \tilde{f} vdx +\int_{B} \langle \tilde{\mathbf{F}}, \nabla v \rangle dx. 
\end{align}
Consider $\rho \in H^{1,p}_{loc}(\mathbb{R}^d) \cap C(\mathbb{R}^d)$, where $\rho(x)>0$ for every $x \in \mathbb{R}^d$, as defined in Theorem \ref{exisinfinvar} and define $\mu$, $\mathbf{B}$ as in \eqref{rhodivfreeb}. Let $(\mathcal{E}^B, D(\mathcal{E}^B))$ and $(L^B, D(L^B))$ 
be as in Theorem \ref{matebisdiri} and Proposition \ref{generprop}, respectively. Let $h:=\frac{\tilde{h}}{\rho} \in L^2(B, \mu)$, $f:=\frac{\tilde{f}}{\rho}\in L^{\frac{2d}{d+2}}(B, \mu)$ and $\mathbf{F}:=\frac{1}{\rho} \tilde{\mathbf{F}} \in L^2(B, \mathbb{R}^d, \mu)$. Now, Proposition \ref{generprop} and \eqref{intidenonb} imply that $v \in D(L^B)$ and
\begin{equation} \label{intidengenre}
\int_{B} \Big( L^B v+cv \Big) h d\mu = \int_{B} f v + \langle \mathbf{F}, \nabla v \rangle d\mu.
\end{equation}
Let $\alpha>0$. Then, $G_{\alpha}^B h \in H^{2,2}(B) \cap H^{1,\frac{2d}{d-2}}_0(B)$ by Theorem \ref{resolregul}. Since $\chi \alpha G_{\alpha}^B h \in H^{2,2}(B) \cap H^{1,\frac{2d}{d-2}}_0(B)$ and has a compact support in $B$, replacing $v$ with $\chi \alpha G_{\alpha}^B h$ in \eqref{intidengenre}, we obtain that
\begin{equation} \label{impotidenge}
- \int_{B} L^B \big( \chi \alpha G_{\alpha}^B h  \big) \cdot h d\mu = \int_{B} c (\chi \alpha G_{\alpha}^B h) h d\mu - \int_{B} f (\chi \alpha G_{\alpha}^B h) - \int_{B} \langle \mathbf{F}, \nabla (\chi \alpha G_{\alpha}^B h) \rangle d\mu.
\end{equation}
Thus, by using Proposition \ref{algproper} and the property that $(\alpha-L^B) G_{\alpha}^B h = h$ on $B$, we find that 
\begin{align}
L^B(\chi G_{\alpha}^B h) &= G_{\alpha}^B h L^B \chi + \chi L^B G_{\alpha}^B h +\langle A \nabla \chi, \nabla G_{\alpha}^B h \rangle + \langle A \nabla G_{\alpha}^B h, \nabla \chi \rangle  \nonumber \\
&=G_{\alpha}^B h L^B \chi + \chi (\alpha G_{\alpha} h  -h)+\langle A \nabla \chi, \nabla G_{\alpha}^B h \rangle + \langle A \nabla G_{\alpha}^B h, \nabla \chi \rangle. \label{prodruleres}
\end{align}
By using \eqref{genproperty} in Proposition \ref{generprop} and \eqref{prodruleres}, we conclude that
\begin{align}
&\mathcal{E}^B \big(\chi \alpha G_{\alpha}^B h, \chi \alpha G_{\alpha}^B h \big)   = -\int_{B} L^B (\chi \alpha G_{\alpha}^B f) \cdot \chi \alpha G_{\alpha}^B h d\mu \nonumber  \\
& \quad  = -\int_{B} \alpha G_{\alpha}^B h \cdot (L^B \chi) \cdot \chi \alpha G_{\alpha}^B h d\mu - \int_{B} \langle A \nabla \chi, \nabla \alpha G_{\alpha}^B h \rangle \chi \alpha G_{\alpha}^B h d\mu- \int_{B} \langle A \nabla \alpha G_{\alpha}^B h, \nabla \chi \rangle \chi \alpha G_{\alpha}^B h d\mu \nonumber  \\
& \quad \qquad \qquad - \alpha \int_{B} (\chi \alpha G_{\alpha}^B h - \chi h) \chi \alpha G_{\alpha}^B h d\mu \nonumber  \\
& \quad  = -\int_{B} \alpha G_{\alpha}^B h \cdot (L^B \chi) \cdot \chi \alpha G_{\alpha}^B h d\mu - \int_{B} \langle A \nabla \chi, \nabla (\chi \alpha G_{\alpha}^B h) \rangle \alpha G_{\alpha}^B h d\mu- \int_{B} \langle A \nabla ( \chi \alpha G_{\alpha}^B h), \nabla \chi \rangle \alpha G_{\alpha}^B h d\mu \nonumber  \\
& \quad  \qquad \qquad  +2\int_{B} \langle A \nabla \chi, \nabla \chi \rangle (\alpha G_{\alpha}^B h)^2d\mu - \alpha \int_{B} (\chi \alpha G_{\alpha}^B h - \chi h) \chi \alpha G_{\alpha}^B h d\mu. \label{alaoglucptnes}
\end{align}
First, by the H\"{o}lder inequality, the $L^2(B, \mu)$-contraction property of $(G^B_{\alpha})_{\alpha>0}$ and the Sobolev inequality (cf. \cite[Theorem 4.8]{EG15}),
\begin{align*}
&-\int_{B} \alpha G_{\alpha}^B h \cdot (L^B \chi) \cdot \chi \alpha G_{\alpha}^B h d\mu \leq \|\alpha G_{\alpha}^B h\|_{L^2(B, \mu)} \| L^B \chi \|_{L^d(B, \mu)} \| \chi \alpha G_{\alpha}^B h \|_{L^{\frac{2d}{d-2}}(B, \mu)}\\
& \qquad \leq \underbrace{\|h\|_{L^2(B, \mu)}  \| L^B \chi \|_{L^d(B, \mu)} K_{d, \rho}}_{=:c_1}\, \mathcal{E}^{B} (\chi \alpha G_{\alpha}^B h,  \chi \alpha G_{\alpha}^B h)^{1/2},
\end{align*}
where $K_{d, \rho}:=\frac{(\max_{\overline{B}} \rho)^{\frac12-\frac{1}{d}}}{\lambda^{1/2} (\min_{\overline{B}} \rho)^{\frac12}} \cdot \frac{2(d-1)}{d-2}$.
Using the Cauchy-Schwartz inequality, the H\"{o}lder inequality, the Sobolev inequality and the $L^2(B, \mu)$-contraction property of $(G^B_{\alpha})_{\alpha>0}$,
\begin{align*}
 &- \int_{B} \langle A \nabla \chi, \nabla (\chi \alpha G_{\alpha}^B h) \rangle \alpha G_{\alpha}^B h d\mu \leq \int_{B} 
 \langle A \nabla \chi, \nabla \chi \rangle^{1/2} \langle A \nabla \chi \alpha  G_{\alpha}^B h, \nabla  \chi \alpha G_{\alpha}^B h \rangle^{1/2}
  | \alpha G_{\alpha}^Bh | d\mu \\
& \qquad \leq (dM)^{1/2} \| \nabla  \chi \|_{L^{\infty}(B)} \mathcal{E}^B(\chi \alpha G_{\alpha}^B h, \chi \alpha G_{\alpha} h)^{1/2} \| \alpha G_{\alpha}^B h \|_{L^2(B, \mu)} \\
& \qquad \leq \underbrace{ (dM)^{1/2} \| \nabla  \chi \|_{L^{\infty}(B)}
\|  h \|_{L^2(B, \mu)} }_{=:c_2}\,\mathcal{E}^B(\chi \alpha G_{\alpha}^B h, \chi \alpha G_{\alpha} h)^{1/2}.
\end{align*}
Likewise, we find that
\begin{align*}
&- \int_{B} \langle A \nabla ( \chi \alpha G_{\alpha}^B h), \nabla \chi \rangle \alpha G_{\alpha}^B h d\mu =- \int_{B} \langle A^T \nabla \chi, \nabla (\chi \alpha G_{\alpha}^B h) \rangle \alpha G_{\alpha}^B h\, d\mu \\
&\qquad \leq c_2 \mathcal{E}^B(\chi \alpha G_{\alpha}^B h, \chi \alpha G_{\alpha} h)^{1/2}.
\end{align*}
Moreover, applying the $L^2(B, \mu)$-contraction property of $(G^B_{\alpha})_{\alpha>0}$,
$$
2\int_{B} \langle A \nabla \chi, \nabla \chi \rangle (\alpha G_{\alpha}^B h)^2d\mu \leq \int_B 2 d M \| \nabla \chi \|^2 (\alpha G_{\alpha}^B h)^2\, d\mu
 \leq \underbrace{ 2 dM \| \nabla \chi \|^2_{L^{\infty}(B)} \| h\|^2_{L^2(B, \mu)}}_{=:c_3}.
$$
Furthermore, through straightforward computation
$$
-\alpha \left( \int_{B} (\alpha G_{\alpha}^B h -h) \alpha G_{\alpha}^B h \cdot \chi^2 d\mu -\int_{B} (\alpha G_{\alpha}^B h -h) h \cdot \chi^2 d\mu
\right)= -\alpha \int_{B} (\alpha G_{\alpha}^B h - h)^2 \chi^2 d\mu \leq 0,
$$
and hence
\begin{align} \label{imporinequ}
-\alpha \int_{B} (\alpha G_{\alpha}^B h -h) \alpha G_{\alpha}^B h \cdot \chi^2 d\mu \leq  -\alpha \int_{B} (\alpha G_{\alpha}^B h -h) h \cdot \chi^2 d\mu.
\end{align}
Since $(\alpha-L^{B}) (\alpha G_{\alpha}^B h) = \alpha h$, we have $-\alpha \Big( \alpha G_{\alpha}^B h-h \Big) = -L^B (\alpha G_{\alpha}^B h)$. Thus, applying Proposition \ref{algproper} and \eqref{impotidenge} to \eqref{imporinequ},  we discover that
\begin{align*}
&-\alpha \int_{B} (\alpha G_{\alpha}^B h -h) \alpha G_{\alpha}^B h \cdot \chi^2 d\mu  \leq -\alpha \int_{B} (\alpha G_{\alpha}^B h -h) h  \chi^2 d\mu = \int_{B} -L^B (\alpha G_{\alpha}^B h)\cdot h \chi^2 d\mu \\
& =\int_{B}-L^B (\chi^2 \alpha G_{\alpha}^B h) \cdot h d\mu + \int_{B} \alpha G_{\alpha}^B h \cdot (L^B \chi^2) h d\mu + \int_{B} \langle A \nabla \chi^2, \nabla \alpha G_{\alpha}^B h \rangle h d\mu +\int_{B} \langle A \nabla \alpha G_{\alpha}^B h ,\nabla \chi^2 \rangle hd\mu \\
&=  \int_{B} c (\chi \alpha G_{\alpha}^B h) h d\mu - \int_{B} f (\chi \alpha G_{\alpha}^B h) d\mu - \int_{B} \langle \mathbf{F}, \nabla (\chi \alpha G_{\alpha}^B h) \rangle d\mu \\
& \quad \qquad  + \int_{B} \alpha G_{\alpha}^B h \cdot (L^B \chi^2) h d\mu + \int_{B} \langle A \nabla \chi^2, \nabla \alpha G_{\alpha}^B h \rangle h d\mu +\int_{B} \langle A \nabla \alpha G_{\alpha}^B h ,\nabla \chi^2 \rangle hd\mu.
\end{align*}
By the H\"{o}lder inequality and the Sobolev inequality, 
\begin{align*}
\int_{B} c (\chi \alpha G_{\alpha}^B h) h d\mu  &\leq \|c\|_{L^d(B, \mu)} \| \chi \alpha G_{\alpha}^B h \|_{L^{\frac{2d}{d-2}}(B, \mu)} \|h\|_{L^2(B, \mu)} \\
&\leq \underbrace{\|c\|_{L^d(B, \mu)}\|h\|_{L^2(B, \mu)} K_{d, \rho} }_{=:c_4}\, \mathcal{E}^B(\chi \alpha G_{\alpha}^B h,  \chi \alpha G_{\alpha}^B h)^{1/2},
\end{align*}
and
\begin{align*}
- \int_{B} f (\chi \alpha G_{\alpha}^B h) d\mu &\leq \|f\|_{L^{\frac{2d}{d+2}}(B, \mu)} \| \chi \alpha G_{\alpha}^B h \|_{L^{\frac{2d}{d-2}}(B, \mu)} \\
&\leq \underbrace{\|f\|_{L^{\frac{2d}{d+2}}(B, \mu)}  K_{d, \rho}} _{=:c_5}  \,\mathcal{E}^B(\chi \alpha G_{\alpha}^B h,  \chi \alpha G_{\alpha}^B h)^{1/2}.
\end{align*}
By the H\"{o}lder inequality,
\begin{align*}
- \int_{B} \langle \mathbf{F}, \nabla (\chi \alpha G_{\alpha}^B h) \rangle d\mu &\leq \| \mathbf{F} \|_{L^2(B, \mu)}  \|\nabla \chi \alpha G_{\alpha}^B h \|_{L^2(B, \mu)} \\
&\leq  \underbrace{ \| \mathbf{F} \|_{L^2(B, \mu)} \lambda^{-1/2}}_{=:c_6}\, \mathcal{E}^B(\chi \alpha G_{\alpha}^B h,  \chi \alpha G_{\alpha}^B h)^{1/2}.
\end{align*}
Meanwhile, by Proposition \ref{algproper} and using the H\"{o}lder inequality, the Sobolev inequality and  the $L^2(B, \mu)$-contraction property of $(G^B_{\alpha})_{\alpha>0}$
\begin{align*}
 &\int_{B} \alpha G_{\alpha}^B h \cdot (L^B \chi^2) h d\mu  =  \int_{B} \alpha G_{\alpha}^B h \cdot (2\chi L^B \chi) h d\mu   + \int_{B} \alpha G_{\alpha}^B h \cdot 2 \langle A \nabla \chi, \nabla \chi \rangle h d\mu \\
 & \leq \| \chi \alpha G_{\alpha}^B h \|_{L^{\frac{2d}{d-2}}(B, \mu)} \| 2 L^B \chi \|_{L^d(B, \mu)} \|h\|_{L^2(B, \mu)} + 2dM\|\nabla \chi \|^2_{L^{\infty}(B)} \| \alpha G_{\alpha}^B h\|_{L^2(B, \mu)} \| h\|_{L^2(B, \mu)} \\
 & \leq  \underbrace{K_{d, \rho} \| 2 L^B \chi \|_{L^d(B, \mu)} \|h\|_{L^2(B, \mu)}}_{=:c_7}  \, \mathcal{E}^B(\chi \alpha G_{\alpha}^B h,  \chi \alpha G_{\alpha}^B h)^{1/2} + \underbrace{ 2dM\|\nabla \chi \|^2_{L^{\infty}(B)}  \|h\|^2_{L^2(B, \mu)}}_{=:c_8} .
\end{align*}
From the Cauchy-Schwartz inequality,  the H\"{o}lder inequality and the Sobolev inequality, and  the $L^2(B, \mu)$-contraction property of $(G^B_{\alpha})_{\alpha>0}$,
\begin{align*}
&\int_{B} \langle A \nabla \chi^2, \nabla \alpha G_{\alpha}^B h \rangle h d\mu = \int_{B} 2\langle A \nabla \chi, \chi \nabla  \alpha G_{\alpha}^B h \rangle h d\mu \\
&\quad = \int_{B} 2\langle A \nabla \chi, \nabla  (\chi \alpha G_{\alpha}^B h) \rangle h d\mu   - \int_{B} 2\langle A \nabla \chi, \nabla \chi \rangle (\alpha G_{\alpha}^B h) h d\mu \\
& \quad  \leq \int_{B} 2\langle A \nabla (\chi \alpha G_{\alpha}^B h), \nabla  (\chi \alpha G_{\alpha}^B h) \rangle^{1/2} \langle A \nabla \chi, \nabla \chi \rangle^{1/2} |h| d\mu  \\
&\qquad \qquad \quad + 2dM \| \nabla \chi \|^2_{L^{\infty}(B)} \| \alpha G_{\alpha}^B h \|_{L^2(B, \mu)} \|h\|_{L^2(B, \mu)} \\
&\quad \leq \underbrace{2(dM)^{1/2} \| \nabla \chi\|_{L^{\infty}(B)} \| h\|_{L^2(B, \mu)}}_{=:c_9} \cdot \mathcal{E}^B(\chi \alpha G_{\alpha}^B h,  \chi \alpha G_{\alpha}^B h)^{1/2} + \underbrace{2dM \| \nabla \chi \|^2_{L^{\infty}(B)} \|h\|^2_{L^2(B, \mu)}}_{=:c_{10}}.
\end{align*}
Similarly, we find that
\begin{align*}
&\int_{B} \langle A \nabla \alpha G_{\alpha}^B h ,\nabla \chi^2 \rangle hd\mu =\int_{B} \langle A^T \nabla \chi^2, \nabla \alpha G_{\alpha}^B h \rangle h d\mu \\
&\quad \leq c_9 \,\mathcal{E}^B(\chi \alpha G_{\alpha}^B h,  \chi \alpha G_{\alpha}^B h)^{1/2} + c_{10}.
\end{align*}
Therefore, \eqref{alaoglucptnes} deduces that
$$
\mathcal{E}^B \big(\chi \alpha G_{\alpha}^B h, \chi \alpha G_{\alpha}^B h \big) \leq C_1 \mathcal{E}^B \big(\chi \alpha G_{\alpha}^B h, \chi \alpha G_{\alpha}^B h \big) ^{1/2}+C_2,
$$
where $C_1= c_1+2c_2+c_4+c_5+c_6+c_7+2c_9$ and $C_2=c_3+c_8+2c_{10}$. Thus, Young's inequality implies that
$$
\mathcal{E}^B \big(\chi \alpha G_{\alpha}^B h, \chi \alpha G_{\alpha}^B h \big)  \leq  C_1^2+2C_2.
$$
Given that $C_1^2+2C_2>0$ is a constant independent of $\alpha>0$, it follows that 
$$
\sup_{\alpha>0} \mathcal{E}^B \big(\chi \alpha G_{\alpha}^B h, \chi \alpha G_{\alpha}^B h \big)  \leq  C_1^2+2C_2.
$$
Using the Banach-Alaoglu theorem, there exist $v \in D(\mathcal{E}^B)$ and
a strictly increasing sequence $(\alpha_k)_{k \geq 1} \subset \mathbb{N}$ such that 
$$
\lim_{k \rightarrow \infty} \chi \alpha_k G_{\alpha_k}^B h  = v, \quad \text{ weakly in $(D(\mathcal{E}^B), \| \cdot \|_{D(\mathcal{E}^B_1)})$},
$$
where $\| \cdot \|_{D(\mathcal{E}^B_1)}$ is defined as $\| w \|_{D(\mathcal{E}^B_1)}:= \mathcal{E}^B(w,w)^{1/2}+\|w\|_{L^2(B, \mu)}$,  $w \in D(\mathcal{E}^B)$. On the other hand, by the strong continuity of $(G_{\alpha}^B)_{\alpha>0}$ on $L^2(B, \mu)$ we have $\lim_{k \rightarrow \infty} \chi \alpha_k G_{\alpha_k}^B h = \chi h$ in $L^2(B, \mu)$. Thus, $\chi h = v \in D(\mathcal{E}^B)=H^{1,2}_0(B)$. Therefore, by Lemma \ref{auxresulem}(i) $\chi \tilde{h} = \rho (\chi h) \in H^{1,2}_0(B)$, as desired.
\end{proof}

\begin{theo} \label{maintheoremai}
Assume that {\bf (Hyp)} holds. Let $x_0 \in \mathbb{R}^d$, $R>0$ and $q \in [2, \infty)$. Suppose that ${\rm div} A \in L_{loc}^d(B_R(x_0), \mathbb{R}^d)$, $\tilde{h} \in L^2_{loc}(B_R(x_0))$, $c \in L^{d}_{loc}(B_R(x_0))$, $\tilde{f} \in L^{\frac{qd}{d+q}}_{loc}(B_R(x_0))$ and $\tilde{\mathbf{F}} \in L^q_{loc}(B_R(x_0), \mathbb{R}^d)$ satisfy
$$
\int_{B_R(x_0)} \Big( \text{\rm div} \big( A \nabla u \big) + \langle \mathbf{H}, \nabla u \rangle +cu \Big) \tilde{h} dx = \int_{B_R(x_0)} \tilde{f} udx +\int_{B_R(x_0)} \langle \tilde{\mathbf{F}}, \nabla u \rangle dx \quad \text{ for all $u \in C_0^{\infty}(B_R(x_0))$}.
$$
Then, $\frac{\tilde{h}}{\rho} \in H_{loc}^{1,p \wedge q}(B_R(x_0))$. In particular, $\tilde{h} \in H^{1,{p \wedge q}}_{loc}(B_R(x_0))$. 
\end{theo}
\begin{proof}
By Theorem \ref{mainresreg}, $\tilde{h} \in H^{1,2}_{loc}(B_R(x_0))$.
Consider an arbitrary $s_0 \in (0, R)$ and select $s \in (s_0, R)$ and $r \in (s, R)$. Write $B:=B_r(x_0)$. Let $\chi \in C_0^{\infty}(B)$ with $\chi \equiv 1$ on $\overline{B}_{s}(x_0)$ and $0 \leq \chi \leq 1$ on $\mathbb{R}^d$. To establish the assertion, it suffices to demonstrate that $\tilde{h} \in H^{1, p \wedge q}(B_{s_0}(x_0))$. 
Let $\rho \in H^{1,p}_{loc}(\mathbb{R}^d) \cap C(\mathbb{R}^d)$ with $\rho(x)>0$ for all $x \in \mathbb{R}^d$ be  as in Theorem \ref{exisinfinvar} and define $\mu$, $\mathbf{B}$ as in \eqref{rhodivfreeb}. Let $(\mathcal{E}^B, D(\mathcal{E}^B))$ and $(L^B, D(L^B))$ 
be as in Theorem \ref{matebisdiri} and Proposition \ref{generprop}, respectively. Define $h := \frac{\tilde{h}}{\rho} \in H^{1,2}(B)$, $f := \frac{\tilde{f}}{\rho}\in L^{\frac{qd}{d+q}}(B)$, and $\mathbf{F} := \frac{1}{\rho} \tilde{\mathbf{F}} \in L^q(B, \mathbb{R}^d)$. Now let $u \in C_0^{\infty}(B_s(x_0))$ be arbitrarily chosen. Then,
\begin{align*}
&-\int_{B} \Big( \text{\rm div} \big( A \nabla u \big) + \langle \mathbf{H}, \nabla u \rangle +cu \Big) (\chi h) d\mu =
-\int_{B} \Big( \text{\rm div} \big( A \nabla u \big) + \langle \mathbf{H}, \nabla u \rangle +cu \Big) h dx \\
&\quad =-\int_{B} \tilde{f} u dx -\int_{B} \langle \tilde{\mathbf{F}}, \nabla u \rangle dx =-\int_{B} f u\, d\mu -\int_{B} \langle \mathbf{F}, \nabla u \rangle \,d\mu.
\end{align*}
Given that $\chi u \in C_0^{\infty}(B_s(x_0)) \subset H^{1,2}_0(B) = D(\mathcal{E}^B)$, it follows from Proposition \ref{generprop} and Theorem \ref{matebisdiri}(iii),
\begin{align*}
& -\int_{B} f u\, d\mu -\int_{B} \langle \mathbf{F}, \nabla u \rangle \,d\mu \\
&\quad =-\int_{B} \Big( \text{\rm div} \big( A \nabla u \big) + \langle \mathbf{H}, \nabla u \rangle +cu \Big) (\chi h) d\mu  =- \int_{B} L^B u \cdot (\chi h) \, d\mu - \int_{B} cu (\chi h) d\mu \\
& \quad  = \mathcal{E}^B (u, \chi h) - \int_{B} cu (\chi h) d\mu  = \int_{B} \langle A \nabla u, \nabla (\chi h) \rangle d\mu- \int_{B} \langle \mathbf{B}, \nabla u \rangle \chi h d\mu -\int_{B} cu(\chi h) d\mu.
\end{align*}
Consequently, for any $\varphi \in C_0^{\infty}(B_s(x_0))$
\begin{align}
& \int_{B_s(x_0)} \langle \rho A^T \nabla h, \nabla \varphi  \rangle dx+ \int_{B_{s}(x_0)} \langle \rho \mathbf{B}, \nabla h \rangle \varphi dx +\int_{B_s(x_0)} -\rho c h \varphi dx \nonumber \\
& \quad = \int_{B_s(x_0)} -\rho f \varphi dx + \int_{B_s(x_0)} \langle -\rho \mathbf{F}, \nabla \varphi \rangle dx. \label{variaidenhp}
\end{align}
Let $\hat{p}:=p \wedge q$. Then, $-\rho \mathbf{B} \in L^{\hat{p}}(B, \mathbb{R}^d)$ and $-\rho c \in L^{\frac{d\hat{p}}{d+\hat{p}}}(B_s(x_0))$, $-\rho f \in L^{\frac{d \hat{p}}{d+\hat{p}}}(B_s(x_0))$ and $-\rho \mathbf{F} \in L^{\hat{p}}(B_{s}(x_0), \mathbb{R}^d)$. By applying \cite[Theorem 1.8.3]{BKRS15} to \eqref{variaidenhp}, we conclude that $h \in H^{1, \hat{p}}(B_{s_0}(x_0))$.  Finally, we establish from Lemma \ref{auxresulem}(ii) that $\tilde{h}=\rho h \in H^{1, p \wedge \hat{p}}(B_{s_0}(x_0))= H^{1, \hat{p}}(B_{s_0}(x_0))$, completing the proof.
\end{proof}

\begin{rem}
In Theorem \ref{maintheoremai}, one can directly obtain from Theorem \ref{mainresreg} and integration by parts that $\tilde{h} \in H^{1,2}_{loc}(B_R(x_0))$ satisfies
\begin{equation} \label{newtypeform}
\int_{B_R(x_0)} \langle A^T \nabla \tilde{h}, \nabla \varphi \rangle - \langle \tilde{h} \mathbf{H}, \nabla \varphi \rangle + c \tilde{h} \varphi dx = \int_{B_R(x_0)} \tilde{f} \varphi\,dx +\int_{B_R(x_0)} \langle \tilde{\mathbf{F}}, \nabla \varphi \rangle dx \quad \text{ for all $u \in C_0^{\infty}(B_R(x_0))$}.
\end{equation}
Applying \cite[Theorem 1.8.3]{BKRS15} to \eqref{newtypeform}, we directly discover that $\tilde{h} \in H^{1,p\wedge q}_{loc}(B_R(x_0))$. Of course, this approach is easier than the proof of Theorem \ref{maintheoremai}, but we mention that the divergence type equation in \eqref{variaidenhp} presents meaningful input for the local regularity of solutions. For instance, in \eqref{variaidenhp} $\frac{\tilde{h}}{\rho}=h$ may have $H^{2, p \wedge q}_{loc}(B_R(x_0))$-regularity under additional conditions on the coefficients (cf. \cite[Theorem 2.2]{KK19}). 
On the other hand, \eqref{newtypeform} may not give this regularity result directly.
\end{rem}

\centerline{}
\noindent
{\bf Proof of Theorem \ref{intromainth}}\\
Consider an arbitrary point $x_0 \in U$. Then, there exists $R_0>0$ such that $\overline{B}_{R_0}(x_0) \subset U$. Choose $R \in (0, R_0)$.
To establish the assertion, it is enough to show that $h \in H^{1, p \wedge q}_{loc}(B_R(x_0))$. Let $\chi \in C_0^{\infty}(B_{R_0}(x_0))$ be such that $\chi \equiv 1$ on $B_R(x_0)$ and $0 \leq \chi \leq 1$ on $\mathbb{R}^d$. Since $a_{ij} \in VMO(B_{R_0}(x_0))$ for all $1 \leq i,j \leq d$, by Lemma \ref{bdextenlem} there exists $\hat{a}_{ij} \in VMO \cap L^{\infty}(\mathbb{R}^d)$ for each $1 \leq i,j \leq d$ such that $\hat{a}_{ij}=a_{ij}$ on $B_{R_{0}}(x_0)$. Let $\hat{A}:=(\hat{a}_{ij})_{1 \leq i,j \leq d}$. We define a matrix of functions $\overline{A} = (\overline{a}_{ij})_{1 \leq i,j \leq d}$ as follows:
$$
\overline{A} = \chi( \hat{A}- \lambda_{B_{R_0}(x_0)} id) + \lambda_{B_{R_0}(x_0)} id,
$$
where $id$ denotes the $d \times d$ identity matrix. Then, first by Proposition \ref{algprovmo}, $\overline{a}_{ij} \in VMO \cap L^{\infty}(\mathbb{R}^d)$ for all $1 \leq i,j \leq d$.  Moreover, $\overline{a}_{ij}(x)=\hat{a}_{ij}(x)=a_{ij}(x)$ for all $x \in B_R(x_0)$ and $1 \leq i,j \leq d$ and
$$
\lambda_{B_{R_0}(x_0)} \| \xi \|^2 \leq \langle \overline{A}(x) \xi, \xi \rangle, \;\; \text{ and } \;\; \max_{1 \leq i,j \leq d} |\overline{a}_{ij} (x)| \leq M_{B_{R_0}(x_0)} + 2 \lambda_{B_{R_0}(x_0)} \quad \text{ for a.e. $x \in \mathbb{R}^d$ and for all $\xi \in \mathbb{R}^d$}. 
$$
Since ${\rm div }\overline{A} ={\rm div} A \in L^d(B_R(x_0), \mathbb{R}^d)$ on $B_R(x_0)$, it follows for any $\varphi \in C_0^{\infty}(B_R(x_0))$ that
\begin{align*}
&{\rm div} (\overline{A} \nabla \varphi) + \langle \mathbf{H}, \nabla \varphi \rangle  = \text{trace}(\overline{A} \nabla^2 \varphi) + \langle {\rm div} \overline{A} +\mathbf{H}, \nabla \varphi \rangle \\
& = \text{trace}(A \nabla^2 \varphi) + \langle {\rm div} A +\mathbf{H}, \nabla \varphi \rangle  = \mathcal{L} \varphi, \quad \text{ on $B_R(x_0)$}.
\end{align*}
Consequently, \eqref{basicintideneq} leads to
\begin{equation} \label{desiredeq}
\int_{B_R(x_0)} \Big({\rm div} (\overline{A} \nabla \varphi) + \langle \mathbf{H}, \nabla \varphi \rangle +c\varphi \Big) h dx =  \int_{B_R(x_0)} f \varphi dx +\int_{B_R(x_0)} \langle \mathbf{F}, \nabla \varphi \rangle dx \quad \text{ for all $u \in C_0^{\infty}(B_R(x_0))$}.
\end{equation}
Observe that the condition {\bf (Hyp)} from Section \ref{2ndsection} remains valid when $A$ is replaced by $\overline{A}$. Thus, by applying Theorem \ref{maintheoremai} to \eqref{desiredeq}, we conclude that the assertion holds.  \\
\text{}\hfill $\square$
\centerline{}
{\bf Proof of Corollary \ref{directcorol}} \\
Let $\mathbf{H}:= \mathbf{G} -{\rm div} A$. Since $\mathbf{G} \in L^p_{loc}(U, \mathbb{R}^d)$ with $p \in (d, \infty)$ and ${\rm div} A\in L^p_{loc}(U, \mathbb{R}^d)$, it follows that $\mathbf{H}\in L^p_{loc}(U, \mathbb{R}^d)$.
Moreover, for any $\varphi \in C_0^{\infty}(U)$, we have 
$$
\mathcal{L} \varphi={\rm trace}(A \nabla^2 \varphi) + \langle \mathbf{G}, \nabla \varphi \rangle \quad \text{ for any $\varphi \in C_0^{\infty}(U)$.}
$$
Therefore, the result follows directly from Theorem \ref{intromainth}. \\
\text{}\hfill $\square$
\centerline{}
{\bf Proof of Corollary \ref{finmaincor}} \\
First by \cite[Theorem 2.1]{BS17}, there exists $h \in L^2_{loc}(U)$ such that $\nu=hdx$. Applying Corollary \ref{directcorol} with $q=p$, $\mathbf{F}=0$ and $f=0$, we find that $h \in H^{1,p}_{loc}(U) \cap C^{0, 1-d/p}_{loc}(U)$, as desired.
\text{}\hfill $\square$
\centerline{}
\begin{exam}
\begin{itemize}
\item[(i)]  
Define a function \(\phi: \mathbb{R}^d \to \mathbb{R}\) by  
\[
\phi(x) := 2 + \|x\|^2 + \cos\left(\ln \ln\left(1 + \frac{1}{\|x\|}\right)\right), \quad x \in \mathbb{R}^d \setminus \{0\}, \quad \phi(0) := 0.
\]  
It has been shown in \cite{L23ID} that \(\phi \in H^{1,d}_{\text{loc}}(\mathbb{R}^d) \cap L^\infty_{\text{loc}}(\mathbb{R}^d)\) with \(\phi \geq 1\) on \(\mathbb{R}^d\), but \(\phi \notin C(\mathbb{R}^d)\). Consequently, \(\nabla \phi \notin L^r_{\text{loc}}(\mathbb{R}^d, \mathbb{R}^d)\) for any \(r \in (1, \infty)\).  

Now, define a matrix of functions \(A = (a_{ij})_{1 \leq i,j \leq d}\) by setting $A(x) = \phi(x) \, \text{id}$, $x \in \mathbb{R}^d$.
Let \(p \in (d, \infty)\), and let \(\mathbf{H} \in L^p_{\text{loc}}(\mathbb{R}^d, \mathbb{R}^d)\) be an arbitrary vector field. Consider the diffusion operator \((\mathcal{L}, C_0^\infty(\mathbb{R}^d))\), defined as follows:  
\[
\mathcal{L}f := \text{\rm div}(A \nabla f) + \langle \mathbf{H}, \nabla f \rangle = \text{\rm trace}(A \nabla^2 f) + \langle \text{\rm div} A + \mathbf{H}, \nabla f \rangle, \quad f \in C_0^\infty(\mathbb{R}^d).
\]  
By Proposition~\ref{limitmatri}, this simplifies to  
\[
\mathcal{L}f = \phi \Delta f + \langle \nabla \phi + \mathbf{H}, \nabla f \rangle, \quad f \in C_0^\infty(\mathbb{R}^d).
\]  
The diffusion coefficient of \(\mathcal{L}\) is \(\phi \, \text{id}\), whose components are not continuous. The drift coefficient of \(\mathcal{L}\) is \(\nabla \phi + \mathbf{H} \in L^d_{\text{loc}}(\mathbb{R}^d, \mathbb{R}^d)\), but \(\nabla \phi + \mathbf{H} \notin L^r_{\text{loc}}(\mathbb{R}^d, \mathbb{R}^d)\) for any \(r \in (d, \infty)\).  Our main result, Theorem~\ref{intromainth}, provides new insights into the regularity of the density of the infinitesimally invariant measure for \((\mathcal{L}, C_0^\infty(\mathbb{R}^d))\). For instance, if a measure \(\nu = h \, dx\), with \(h \in L^2_{\text{loc}}(\mathbb{R}^d)\), is an infinitesimally invariant measure for \((\mathcal{L}, C_0^\infty(\mathbb{R}^d))\), i.e.
\[
\int_{\mathbb{R}^d} \mathcal{L}f \, d\nu = 0, \quad \text{for all } f \in C_0^\infty(\mathbb{R}^d),
\]  
then Theorem~\ref{intromainth} implies that \(h \in H^{1,p}_{\text{loc}}(\mathbb{R}^d) \cap C^{0,1-d/p}_{loc}(\mathbb{R}^d)\). This result may not be derived using the results in \cite{BKRS15, BS17}.  

\item[(ii)]  
Let \(\eta: \mathbb{R} \to \mathbb{R}\) be a continuous function that is not absolutely continuous (e.g., \(\eta(x) := x \sin\left(\frac{1}{x}\right)\) for \(x > 0\) and \(\eta(x) := 0\) for \(x \leq 0\)). Define \(\phi(x) := \exp(\eta(x))\) for \(x \in \mathbb{R}^d\). Then, \(\phi\) is a continuous function with \(\phi(x) > 0\) for all \(x \in \mathbb{R}^d\), but \(\phi\) is not absolutely continuous, so that \(\phi \notin H^{1,1}_{\text{loc}}(\mathbb{R})\).  

Define a matrix \(A = (a_{ij})_{1 \leq i,j \leq d}\) by  
\[
a_{ij} = 
\begin{cases} 
0 & \text{if } i \neq j, \\ 
\phi(x_{i+1}) & \text{if } i = j \text{ and } i < d, \\ 
\phi(x_1) & \text{if } i = j = d.
\end{cases}
\]  
Then \(a_{ii} \notin H^{1,1}_{\text{loc}}(\mathbb{R}^d)\) for all \(1 \leq i \leq d\), but it holds that ${\rm div} A =0$. Let \(\tilde{\mathbf{H}} \in L^p_{\text{loc}}(\mathbb{R}^d, \mathbb{R}^d)\) with $p \in (d, \infty)$ be given. Consinder the partial differential operator \((L, C_0^\infty(\mathbb{R}^d))\) defined by  
\[
Lf = \frac{1}{2} \text{\rm trace}(A \nabla^2 f) + \langle \tilde{\mathbf{H}}, \nabla f \rangle, \quad f \in C_0^\infty(\mathbb{R}^d).
\]  
If the measure \(\nu = h \, dx\), with \(h \in L^2_{\text{loc}}(\mathbb{R}^d)\), is an infinitesimally invariant measure for \((L, C_0^\infty(\mathbb{R}^d))\), i.e.
$$
\int_{\mathbb{R}^d} Lf \, d\nu = 0, \quad \text{for all } f \in C_0^\infty(\mathbb{R}^d),
$$ 

then Corollary~\ref{directcorol} implies that \(h \in H^{1,p}_{\text{loc}}(\mathbb{R}^d) \cap C^{0,1-d/p}_{loc}(\mathbb{R}^d)\). This result also may not be derived using the results in \cite{BKRS15, BS17}. It is remarkable that the density $h$ of $\nu$ is locally H\"{o}lder continuous even though the components $a_{ii}$ of $A$ are not locally H\"{o}lder continous.
\end{itemize}
\end{exam}

\section{Auxiliary results} \label{sec5}
\begin{prop} \label{algprovmo}
Let $d \geq 1$. If $f, g \in VMO \cap L^{\infty}(\mathbb{R}^d)$, then $fg \in VMO \cap L^{\infty}(\mathbb{R}^d)$.
\end{prop}
\begin{proof}
Let $f,g \in VMO \cap L^{\infty}(\mathbb{R}^d)$. Then, there exist positive continuous functions $\omega_1$ and $\omega_2$ on $[0, \infty)$ with $\omega_1(0)=\omega_2(0)=0$ such that
$$
\sup_{z \in \mathbb{R}^d, r < R} r^{-2d} \int_{B_r(z)} \int_{B_r(z)} |f(x) - f(y)| \, dx \, dy \leq \omega_1(R), \quad \text{for every $R > 0$}
$$
and
$$
\sup_{z \in \mathbb{R}^d, r < R} r^{-2d} \int_{B_r(z)} \int_{B_r(z)} |g(x) - g(y)| \, dx \, dy \leq \omega_2(R), \quad \text{for every $R > 0$}.
$$
Let $z \in \mathbb{R}^d$ and $R>0$. Choose $r \in (0, R)$. Then for each $x,y \in B_r(z)$,
\begin{align*}
|f(x)g(x)-f(y)g(y)| &\leq |f(x)g(x)-f(x)g(y)|+|f(x)g(y)-f(y)g(y)| \\
&\leq \|f\|_{L^{\infty}(\mathbb{R}^d)} |g(x)-g(y)| + \|g\|_{L^{\infty}(\mathbb{R}^d)} |f(x)-f(y)|.
\end{align*}
Therefore, it follows that
\begin{align*}
\sup_{z \in \mathbb{R}^d, r < R}r^{-2d}\int_{B_r(z)} \int_{B_r(z)} |f(x)g(x) - f(y)g(y)| \, dx \, dy  \leq \|f\|_{L^{\infty}(\mathbb{R}^d)}  \omega_2(R)  + \|g\|_{L^{\infty}(\mathbb{R}^d)}  \omega_1(R),
\end{align*}
as desired.
\end{proof}

\begin{lem} \label{bdextenlem}
Let $B$ be an open ball in $\mathbb{R}^d$ with $d \geq 1$ and $f \in VMO(B) \cap L^{\infty}(B)$. Then, there exists $\hat{f} \in VMO \cap L^{\infty}(\mathbb{R}^d)$ such that $\hat{f}|_B=f$.
\end{lem}
\begin{proof}
Since $f \in VMO(B)$, there exists $\tilde{f} \in VMO$ such that $\tilde{f}|_B = f$. In particular, there exists a positive continuous function $\omega$ on $[0, \infty)$ such that
$$
\sup_{z \in \mathbb{R}^d, r < R} r^{-2d} \int_{B_r(z)} \int_{B_r(z)} |\tilde{f}(x) - \tilde{f}(y)| \, dx \, dy \leq \omega(R), \quad \text{ for every $R>0$}.
$$
Let $M:=\|f\|_{L^{\infty}(B)}$ and choose $\eta \in C_0^{\infty}(\mathbb{R})$ such that 
$\eta(t)=t$ for any $t \in [-M-1, M+1]$ and $\eta(t)=0$ for any $t \in \mathbb{R} \setminus (-2M-2, 2M+2)$. Define $\hat{f}:= \eta \circ \tilde{f}$. Then, $\hat{f} \in L^{\infty}(\mathbb{R}^d)$ and $\hat{f}|_B=f$. Now, let $z \in \mathbb{R}^d$ and $R>0$. Choose $r \in (0, R)$. Then for each $x,y \in B_r(z)$, by the mean value theorem
$$
|\hat{f}(x)-\hat{f}(y)|=|(\eta \circ \tilde{f}) (x) - (\eta \circ \tilde{f})(y)| \leq \| \eta'\|_{L^{\infty}(\mathbb{R})} |\tilde{f}(x)-\tilde{f}(y)|.
$$
Therefore,  for every $R > 0$
$$
\sup_{z \in \mathbb{R}^d, r < R} r^{-2d} \int_{B_r(z)} \int_{B_r(z)} |\hat{f}(x) - \hat{f}(y)| \, dx \, dy \leq \| \eta'\|_{L^{\infty}(\mathbb{R})}  \omega(R),
$$
and hence $\hat{f} \in VMO$, as wished.
\end{proof}

\begin{prop}
Let $U$ be an open subset of $\mathbb{R}^d$ with $d \geq 1$ and $f,g \in VMO_{loc}(U) \cap L^{\infty}_{loc}(U)$. Then, $fg \in VMO_{loc}(U) \cap L^{\infty}_{loc}(U)$.
\end{prop}
\begin{proof}
Let $B$ be an open ball in $\mathbb{R}^d$ with $\overline{B} \subset U$. Then, $f,g \in VMO(B) \cap L^{\infty}(B)$. By Lemma \ref{bdextenlem}, 
there exists $\hat{f}, \hat{g} \in VMO \cap L^{\infty}(\mathbb{R}^d)$ such that $\hat{f}|_B = f$ and $\hat{g}|_B = g$.
By Proposition \ref{algprovmo}, $\hat{f} \hat{g} \in VMO \cap L^{\infty}(\mathbb{R}^d)$. Since, $\hat{f} \hat{g}|_{B} = fg$, the assertion follows.
\end{proof}

\begin{prop} \label{limitmatri}
Let $U$ be an open subset of $\mathbb{R}^d$ with $d \geq 2$, $B = (b_{ij})_{1 \leq i,j \leq d}$ be a matrix of locally integrable functions on $U$ and $\mathbf{E}=(e_1,\ldots, e_d) \in L^1_{loc}(U, \mathbb{R}^d)$ be such that ${\rm{div}} B = \mathbf{E}$.  Let $\Phi=(\phi_1, \ldots, \phi_d) \in C_0^{\infty}(U)^d$ be arbitrary and $V$ be an open set in $\mathbb{R}^d$  such that 
$\text{\rm supp} \, \phi_j \subset V \subset \overline{V} \subset U$ for any $j=1, \ldots, d$. Let $B_n=(b^n_{ij})_{1 \leq i,j \leq d, n \geq 1}$ be a sequence of matrices of functions in $C^1(\overline{V})$ such that $\lim_{n \rightarrow \infty} b^n_{ij} =b_{ij}$ weakly in $L^1(V)$ for each $1 \leq i,j \leq d$. Then,
\begin{align*}
\int_{U} \langle \text{\rm div} B, \Phi \rangle dx= \lim_{n \rightarrow \infty}  \int_{U} \sum_{j=1}^d \Big(\sum_{i=1}^d \partial_i b^n_{ij}\Big)  \phi_j dx.
\end{align*}
\end{prop}
\begin{proof}
By the definition of ${\rm div} B$ and the conditions above, we obtain that
\begin{align*}
\int_{U} \langle \text{\rm div} B, \Phi \rangle dx&= \int_{U} \sum_{j=1}^d e_j \phi_j \, dx=-\int_{U} \sum_{i,j=1}^d b_{ij} \partial_i \phi_j \, dx = \lim_{n \rightarrow \infty} -\int_{U} \sum_{i,j=1}^d b^n_{ij} \partial_i \phi_j \, dx\\
&= \lim_{n \rightarrow \infty} \int_{U} \sum_{j=1}^d \Big(\sum_{i=1}^d \partial_i b^n_{ij}\Big)  \phi_j dx,
\end{align*}
as desired.
\end{proof}

\begin{prop} \label{limitvecfi}
Let $U$ be an open subset of $\mathbb{R}^d$ with $d \geq 2$, $\mathbf{F} \in L^1_{loc}(U, \mathbb{R}^d)$ and $f \in L^1_{loc}(U)$ be such that ${\rm div} \mathbf{F} =f$. Let $\phi \in C_0^{\infty}(U)$ be arbitrary and $V$ be an open set in $\mathbb{R}^d$ such that 
$\text{\rm supp} \, \phi \subset V \subset \overline{V} \subset U$. Let $(\mathbf{F}_n)_{n \geq 1}$ be a sequence of vector fields in $C^1(\overline{V}, \mathbb{R}^d)$ such that $\lim_{n \rightarrow \infty} \mathbf{F}_n =\mathbf{F}$ weakly in $L^1(V, \mathbb{R}^d)$. Then,
\begin{align*}
\int_{U} \text{\rm div} \mathbf{F} \cdot \phi  \,dx= \lim_{n \rightarrow \infty}  \int_{U} \text{\rm div} \mathbf{F}_n \cdot \phi  \,dx.
\end{align*}
\end{prop}
\noindent
\begin{proof}
By the definition, we directly get
$$
\int_{U} {\rm div} \mathbf{F} \cdot \phi\,dx = -\int_{U} \langle \mathbf{F}, \nabla \phi \rangle dx = \lim_{n \rightarrow \infty} -\int_{U} \langle \mathbf{F}_n, \nabla \phi \rangle dx = \lim_{n \rightarrow \infty} \int_{U} {\rm div} \mathbf{F}_n \cdot \phi\,dx.
$$
\end{proof}

\begin{prop} \label{convdivnon}
Let $d \geq 2$, $U$ be an open subset of $\mathbb{R}^d$, $u \in H^{2,1}_{loc}(U)$ with $\nabla u \in L^2_{loc}(U, \mathbb{R}^d)$, and $A=(a_{ij})_{1 \leq i,j \leq d}$ be a (possibly non-symmetric)
matrix of functions in $L^{\infty}_{loc}(U)$ with ${\rm div} A \in L_{loc}^2(U, \mathbb{R}^d)$. Then, 
$$
{\rm div } (A \nabla u) = {\rm trace}(A \nabla^2 u) + \langle {\rm  div} A, \nabla u \rangle \;\; \text{ in $U$}.
$$
\end{prop}
\begin{proof}
Let $\varphi \in C_0^{\infty}(U)$ be fixed. Choose an open set $V$ in $\mathbb{R}^d$ such that  $\text{supp} (\varphi) \subset V \subset \overline{V} \subset U$.
Through a mollification, let $A_n=(a^n_{ij})_{1 \leq i,j \leq d, n \geq 1}$ be a sequence of matrices of functions in $C_0^{\infty}(\overline{V})$ such that
$$
\lim_{n \rightarrow \infty} a^n_{ij}=a_{ij}, \quad \text{ in $L^1(V)$\;\; \text{ for all $1 \leq i,j \leq d$}}.
$$
By a mollification and Sobolev's embedding, there exists $(u_m)_{m \geq 1} \subset C^{\infty}(\overline{V})$ such that 
$$
\lim_{m \rightarrow \infty} u_m =u \; \text{ in  } \; H^{2,1}(V), \quad \quad \lim_{m \rightarrow \infty}\nabla u_m = \nabla u\; \text{  in } L^{2}(V, \mathbb{R}^d).
$$
By Proposition \ref{limitmatri},
\begin{align*} 
\int_{U} \langle A \nabla u, \nabla \varphi \rangle dx &= \lim_{m \rightarrow \infty} \int_{U} \langle A \nabla u_m, \nabla \varphi \rangle dx =  \lim_{m \rightarrow \infty} \lim_{n \rightarrow \infty} \int_{U} \langle A_n \nabla u_m, \nabla \varphi \rangle dx \\
&=\lim_{m \rightarrow \infty} \lim_{n \rightarrow \infty} -\int_{U} \text{trace} (A_n \nabla^2 u_m) \, \varphi + \sum_{j=1}^d\Big(\sum_{i=1}^d \partial_i a^n_{ij} \Big) \varphi \partial_j u_m \, dx  \\
&= \lim_{m \rightarrow \infty} -\int_{U} \text{trace} (A \nabla^2 u_m) \, \varphi + \langle \text{div} A, \varphi \nabla u_m \rangle dx \\
&=-\int_{U} \Big( \text{trace} (A \nabla^2 u)  + \langle \text{div} A, \nabla u \rangle \Big) \varphi dx.
\end{align*}
Since $\varphi \in C_0^{\infty}(U)$ is arbitrarily chosen, the assertion follows.
\end{proof}

\begin{prop} \label{intbypartform}
Let $d \geq 2$, $U$ be an open subset of $\mathbb{R}^d$, $\rho \in H^{1,2}_{loc}(U) \cap L^{\infty}_{loc}(U)$, and $A=(a_{ij})_{1 \leq i,j \leq d}$ be a matrix of functions in $L^{\infty}_{loc}(U)$ with ${\rm div} A \in L_{loc}^2(U, \mathbb{R}^d)$. Then, 
\begin{equation} \label{solvingpr}
\text{\rm div} (\rho A) =\rho \,\text{\rm div} A +A^T\nabla \rho \in L_{loc}^2(U, \mathbb{R}^d).
\end{equation}
In particular, if $u \in H^{2,1}_{loc}(U)$ with $\nabla u \in L^2_{loc}(U)$, then
$$
\int_{U} \langle \rho A \nabla u, \nabla \varphi \rangle\,dx = -\int_{U} {\rm trace} (\rho A \nabla^2 u) + \langle \rho \,\text{\rm div} A +A^T\nabla \rho, \nabla u \rangle \,dx, \quad \text{ for all $\varphi \in C_0^{\infty}(U)$}.
$$
\end{prop}
\begin{proof}
Let $\Phi=(\phi_1, \ldots, \phi_d) \in C_0^{\infty}(U)^d$ be fixed. Choose an open set $V$ in $\mathbb{R}^d$ such that  $\text{supp} (\phi_i) \subset V \subset \overline{V} \subset U$ for all $i=1, \ldots, d$. Through a mollification, choose a sequence of matrices of functions $(a_{n, ij})_{n \geq 1, 1\leq i,j \leq d}$ in $C^{\infty}(\overline{V})$ such that $\lim_{n \rightarrow \infty} a_{n, ij} = a_{ij}$ in $L^1(V)$ for all $1 \leq i,j \leq d$.
Then, by Proposition \ref{limitmatri}
\begin{align*}
& \int_{U} \sum_{i,j=1}^d  \rho  a_{ij} \, \partial_i \phi_j \, dx =\lim_{m \rightarrow \infty} \lim_{n \rightarrow \infty} \int_{U} \sum_{i,j=1}^d  \rho_m a_{n, ij} \, \partial_i \phi_j \, dx \\
&= 
\lim_{m \rightarrow \infty} \lim_{n \rightarrow \infty}-\int_{U} \sum_{i,j=1}^d \phi_j a_{n,ij} \partial_i \rho_m +\rho_m \phi_j \partial_i a_{n,ij} \,dx \\
&=\lim_{m \rightarrow \infty} -\int_{U} \sum_{i,j=1}^d \phi_j a_{ij} \partial_i \rho_m + \langle {\rm div A}, \rho_m \Phi \rangle \,dx  \\
&= -\int_{U} \langle A^T \nabla \rho+ \rho\, {\rm div} A, \Phi \rangle \,dx.
\end{align*}
Thus, \eqref{solvingpr} follows. Now, let $u \in H^{2,2}_{loc}(U)$ and $\varphi \in C_0^{\infty}(U)$.
As a consequence of Proposition \ref{convdivnon} and \eqref{solvingpr}, we obtain that
\begin{align*}
\int_{U} \langle \rho A \nabla u, \nabla \varphi \rangle dx &= -\int_{U} \Big(\text{trace} (\rho A \nabla^2 u) + \langle \text{div}(\rho A), \nabla u \rangle \Big) \varphi \,dx  \\
&=-\int_{U} \Big(\text{trace} (\rho A \nabla^2 u) + \langle \rho \,\text{\rm div} A +A^T\nabla \rho , \nabla u \rangle \Big) \varphi \,dx, 
\end{align*}
as desired.
\end{proof}

\begin{prop} \label{propdruled}
Let $d \geq 2$, $U$ be an open subset of $\mathbb{R}^d$, $u \in H^{1,2}_{loc}(V)$ and $\mathbf{F} \in L_{loc}^2(U, \mathbb{R}^d)$ be such that $\text{\rm div} \mathbf{F} \in L_{loc}^{2}(U)$. Then, 
$$
{\rm div}(u \mathbf{F}) = \langle \nabla u, \mathbf{F} \rangle + u \cdot  {\rm div} \mathbf{F} \;\; \text{\; \rm in $U$}.
$$
\end{prop}
\begin{proof}
Let $\varphi \in C_0^{\infty}(U)$ be arbitrarily given. Choose an open set $V$ in $\mathbb{R}^d$ such that  $\text{supp} (\varphi) \subset V \subset \overline{V} \subset U$. Observe that there exists a sequence of vector fields $(\mathbf{F}_n)_{n \geq 1} \subset C_0^{\infty}(V, \mathbb{R}^d)$ such that $\lim_{n \rightarrow \infty} \mathbf{F}_n = \mathbf{F}$ in $L^2(V, \mathbb{R}^d)$. Moreover, through a mollification, there exists a sequence of functions $(u_m)_{m \geq 1} \subset C^{\infty}(\overline{V})$ such that $\lim_{m \rightarrow \infty} u_m = u$ in $H^{1,2}(V)$. Thus, we obtain from Proposition \ref{limitvecfi}
that
\begin{align*}
\int_{U}\langle u \mathbf{F}, \nabla \varphi \rangle dx &= \lim_{m \rightarrow \infty} \lim_{n \rightarrow \infty}  \int_{U}\langle u_m \mathbf{F}_n, \nabla \varphi \rangle dx= \lim_{m \rightarrow \infty} \lim_{n \rightarrow \infty} - \int_{U}\langle \nabla u_m, \mathbf{F}_n  \rangle \varphi + u_m \varphi \,{\rm div} {\bf F}_n dx \\
& = \lim_{m \rightarrow \infty} -\int_{U} \Big(\langle \nabla u_m, \mathbf{F} \rangle + u_m {\rm div} \mathbf{F}\Big) \varphi \,dx = -\int_{U} \Big( \langle \nabla u, \mathbf{F} \rangle + u\cdot {\rm div} \mathbf{F} \Big) \varphi \,dx.
\end{align*}
Thus, the assertion follows.
\end{proof}

\begin{prop} \label{varphidefn}
Let $\eta$ be a standard mollifier on $\mathbb{R}$, i.e.
$$
\eta(t):= 
\begin{cases} 
\frac{1}{\int_{(-1,1)} e^{\frac{1}{s^2-1}} \, ds } e^{\frac{1}{t^2-1}}, & \text{if } |t| < 1, \\
0, & \text{if } |t| \geq 1.
\end{cases}
$$ 
Let $\varepsilon>0$, $\eta_{\varepsilon}(t):=\frac{1}{\varepsilon} \eta\left( \frac{t}{\varepsilon} \right)$, and $\psi_{\varepsilon}(t):= \left( t \wedge (1+\varepsilon) \right) \vee (-\varepsilon)$, $t \in \mathbb{R}$. 
Define $\varphi_{\varepsilon} \in C^{\infty}(\mathbb{R})$ given by
$$
\varphi_{\varepsilon}(t):= \eta_{\frac{\varepsilon}{2}} \ast \psi_{\varepsilon}(t)= \int_{-\frac{\varepsilon}{2}}^{\frac{\varepsilon}{2}} \eta_{\frac{\varepsilon}{2}} (s) \psi_{\varepsilon}(t-s) ds, \quad t \in \mathbb{R}.
$$
Then, $\varphi_{\varepsilon}(t) \in [-\varepsilon, 1+\varepsilon]$ for all $t \in \mathbb{R}$ and
$$
0 \leq \varphi_{\varepsilon}(t_2)-\varphi_{\varepsilon}(t_1) \leq t_2 -t_1, \quad \text{ for all $t_1, t_2 \in \mathbb{R}$ with $t_1 \leq t_2$},
$$
so that $\varphi'_{\varepsilon}(t) \in [0,1]$ for all $t \in \mathbb{R}$.
Moreover, $\varphi_{\varepsilon}(t)=t$ for all $t \in [-\frac{\varepsilon}{2},1+\frac{\varepsilon}{2}]$, $\varphi_{\varepsilon}(t) =-\varepsilon$ for all $t \in (-\infty, -\frac{3\varepsilon}{2}]$ and $\varphi_{\varepsilon}(t) = 1+\varepsilon$ for all $t \in [1+\frac{3\varepsilon}{2}, \infty)$.
\end{prop}
\begin{proof}
Let $\varepsilon>0$ be given. Since $\int_{-\frac{\varepsilon}{2}}^{\frac{\varepsilon}{2}} \eta_{\frac{\varepsilon}{2}} (s) ds =1$ and $-\varepsilon \leq \psi_{\varepsilon}(t) \leq 1+\varepsilon$ for all $t \in \mathbb{R}$, we get
$$
\varphi_{\varepsilon}(t)=\int_{-\frac{\varepsilon}{2}}^{\frac{\varepsilon}{2}} \eta_{\frac{\varepsilon}{2}} (s) \psi_{\varepsilon}(t-s) ds \in [-\varepsilon, 1+\varepsilon], \quad \text{ for all $t \in \mathbb{R}$}.
$$
Moreover, since $0 \leq \psi_{\varepsilon}(\tau_2) - \psi_{\varepsilon}(\tau_1)   \leq \int_{\tau_1}^{\tau_2} |\psi'_{\varepsilon}(u)| du \leq  \tau_2-\tau_1$ for all $\tau_1, \tau_2 \in \mathbb{R}$ with $\tau_1 \leq \tau_2$, we obtain that for each $t_1, t_2 \in \mathbb{R}$ with $t_1 \leq t_2$,
$$
\varphi_{\varepsilon}(t_2) - \varphi_{\varepsilon}(t_1) = \int_{-\frac{\varepsilon}{2}}^{\frac{\varepsilon}{2}} \eta_{\frac{\varepsilon}{2}} (s) \Big( \psi_{\varepsilon}(t_2-s)  - \psi_{\varepsilon}(t_1-s) \Big) ds \in [0, t_2-t_1].
$$
Observe that if $t \in [-\frac{\varepsilon}{2},1+\frac{\varepsilon}{2}]$ and $s \in [-\frac{\varepsilon}{2}, \frac{\varepsilon}{2}]$, then $t-s \in [-\varepsilon, 1+\varepsilon]$, so that $\psi_{\varepsilon}(t-s)=t-s$.
Thus, for each $t \in [-\frac{\varepsilon}{2},1+\frac{\varepsilon}{2}]$ we have
$$
\varphi_{\varepsilon}(t) = \int_{-\frac{\varepsilon}{2}}^{\frac{\varepsilon}{2}} \eta_{\frac{\varepsilon}{2}} (s) \psi_{\varepsilon}(t-s) ds =t \int_{-\frac{\varepsilon}{2}}^{\frac{\varepsilon}{2}} \eta_{\frac{\varepsilon}{2}} (s)ds-\int_{-\frac{\varepsilon}{2}}^{\frac{\varepsilon}{2}} s\eta_{\frac{\varepsilon}{2}} (s)ds = t.
$$
Likewise, $\varphi_{\varepsilon}(t) =-\varepsilon\int_{-\frac{\varepsilon}{2}}^{\frac{\varepsilon}{2}} \eta_{\frac{\varepsilon}{2}} (s) ds = -\varepsilon$ for all $t \in (-\infty, -\frac{3\varepsilon}{2}]$ and $\varphi_{\varepsilon}(t) =(1+\varepsilon)\int_{-\frac{\varepsilon}{2}}^{\frac{\varepsilon}{2}} \eta_{\frac{\varepsilon}{2}} (s) ds=1+\varepsilon$ for all $t \in [1+\frac{3\varepsilon}{2}, \infty)$, as desired.
\end{proof}

\begin{prop} \label{auxpropfir}
Let $\varepsilon>0$ and $\varphi_{\varepsilon}$ be a function defined as in Proposition \ref{varphidefn}. Then, the following hold:
\begin{itemize}
\item[(i)]
For each $t \in \mathbb{R}$ and $\varepsilon>0$, $|\varphi_{\varepsilon}(t)| \leq 1+\varepsilon$. Moreover, $\lim_{\varepsilon \rightarrow 0+} \varphi_{\varepsilon}(t)=t^+ \wedge 1$ for each $t \in \mathbb{R}$. 
\item[(ii)]
For each $t \in \mathbb{R}$ and $\varepsilon>0$, $|\varphi'_{\varepsilon}(t)| \leq 1$. Moreover, $\lim_{\varepsilon \rightarrow 0+} \varphi'_{\varepsilon}(t)=1_{[0,1]}(t)$ for each $t \in \mathbb{R}$. 
\end{itemize}
\end{prop}
\begin{proof}
(i) The first statement follows from Proposition \ref{varphidefn}.
Let $t<0$. Choose $\varepsilon_0>0$ so that $t<-\frac{3\varepsilon_0}{2}$. Then, if $\varepsilon \in (0, \varepsilon_0)$, then $\varphi_{\varepsilon}(t)=-\varepsilon$, so that  $\lim_{\varepsilon \rightarrow 0+}\varphi_{\varepsilon}(t)=0$. Next, let $t \in [0,1]$. Then, $\varphi_{\varepsilon}(t)=t$ for any $\varepsilon>0$, so that $\lim_{\varepsilon \rightarrow 0+} \varphi_{\varepsilon}(t)=t$. Finally, let $t>1$. Choose $\varepsilon_1>0$ so that $1+\frac{3\varepsilon_1}{2}<t$. Then, if $\varepsilon \in (0, \varepsilon_1)$, then $\varphi_{\varepsilon}(t)=1+\varepsilon$, so that  $\lim_{\varepsilon \rightarrow 0+}\varphi_{\varepsilon}(t)=1$. Thus, $\lim_{\varepsilon \rightarrow 0+} \varphi_{\varepsilon}(t)=t^+ \wedge 1$ for each $t \in \mathbb{R}$. \\ \\
(ii) The first statement follows from Proposition \ref{varphidefn}. Let $t<0$.  Choose $\varepsilon_0>0$ so that $t<-\frac{3\varepsilon_0}{2}$. Then, if $\varepsilon \in (0, \varepsilon_0)$, then $\varphi'_{\varepsilon}(t)=0$, so that $\lim_{\varepsilon \rightarrow 0+} \varphi'_{\varepsilon}(t)=0$.
Next, let $t \in [0,1]$. Then, $\varphi_{\varepsilon}'(t)=1$ for any $\varepsilon>0$, so that $\lim_{\varepsilon \rightarrow 0+} \varphi_{\varepsilon}'(t)=1$. Finally, let $t>1$. Choose $\varepsilon_1>0$ so that $1+\frac{3\varepsilon_1}{2}<t$. Then, if $\varepsilon \in (0, \varepsilon_1)$, then $\varphi'_{\varepsilon}(t)=0$, so that $\lim_{\varepsilon \rightarrow 0+} \varphi'_{\varepsilon}(t)=0$. In conclusion,
$\lim_{\varepsilon \rightarrow 0+} \varphi'_{\varepsilon}(t)=1_{[0,1]}(t)$ for each $t \in \mathbb{R}$. 
\end{proof}

\begin{prop} \label{auxprophi}
Let $\varepsilon>0$ be given. Then, there exists $\Phi_{\varepsilon} \in C^{\infty}(\mathbb{R})$ such that $\Phi_{\varepsilon}(t) =0$ for all $t \in (-\infty, 1]$, $\Phi'_{\varepsilon}(t)=1$ for all $t \in [1+\varepsilon, \infty)$, and $\Phi'_{\varepsilon}(t) \in [0,1]$ and $\Phi''_{\varepsilon}(t) \geq 0$  for all $t \in \mathbb{R}$. In particular, $\lim_{\varepsilon \rightarrow 0+}\Phi'_{\varepsilon}(t) = 1_{(1, \infty)}(t)$ for each $t \in \mathbb{R}$.
\end{prop}
\begin{proof}
Let $\varepsilon>0$ and $\eta_{\varepsilon} \in C^{\infty}(\mathbb{R})$ be defined as in Proposition \ref{varphidefn}. Then, $\eta_{\varepsilon}(t)=0$ for all $t \in \mathbb{R} \setminus (-\varepsilon, \varepsilon)$ and $\int_{\mathbb{R}} \eta_{\varepsilon}(s) ds = 1$. Define $\zeta_{\varepsilon}$ given by
$$
\zeta_{\varepsilon}(t):= \int_{-\infty}^{t} \eta_{\frac{\varepsilon}{2}}\left(s-1-\frac{\varepsilon}{2}\right) ds, \quad t \in \mathbb{R}.
$$
Then, $\zeta_{\varepsilon}\in C^{\infty}(\mathbb{R})$ is an increasing function satisfying that $\zeta_{\varepsilon}(t)=0$ for all $t \in (-\infty, 1]$ and $\zeta_{\varepsilon}(t)=1$ for all $t \in [1+\varepsilon, \infty)$. Now, define
$$
\Phi_{\varepsilon}(t):= \int_{-\infty}^{t} \zeta_{\varepsilon}(s) ds, \quad t \in \mathbb{R}.
$$
Then, $\Phi_{\varepsilon}$ is our desired one. Indeed, let $t \in (-\infty, 1]$. Then, $\Phi'_{\varepsilon}(t)=\zeta_{\varepsilon}(t)=0$, so that $\lim_{\varepsilon \rightarrow 0+} \Phi'_{\varepsilon}(t) =0$. Finally, let $t \in (1, \infty)$. Then, choose $\varepsilon_0>0$ so that $1+\varepsilon_0<t$. Then, if $\varepsilon \in (0, \varepsilon_0)$, then $\Phi'_{\varepsilon}(t) = \zeta_{\varepsilon}(t) =1$, so that $\lim_{\varepsilon \rightarrow 0+} \Phi'_{\varepsilon}(t)=1$. Therefore, $\lim_{\varepsilon \rightarrow 0+}\Phi'_{\varepsilon}(t) = 1_{(1, \infty)}(t)$ for each $t \in \mathbb{R}$. Finally, since $\Phi''_{\varepsilon}(t) = \eta_{\frac{\varepsilon}{2}}\left(t-1-\frac{\varepsilon}{2}\right)  \geq 0$ for all $t \in \mathbb{R}$ and $\varepsilon > 0$, the assertion follows.
\end{proof}

\begin{lem} \label{auxresulem}
Let $B$ be an open ball in $\mathbb{R}^d$ with $d \geq 3$. Then, The following statements hold:
\begin{itemize}
\item[(i)]
If $\bar{\rho} \in H^{1,p}(B)$ with $p \in (d, \infty)$ and $\bar{h} \in H_0^{1,2}(B)$, then $\bar{\rho} \bar{h} \in H_0^{1,2}(B)$.

\item[(ii)]
If $\bar{\rho} \in H^{1,p}(B)$ with $p \in (d, \infty)$ and $\bar{h} \in H^{1,q}(B)$ with $q \in [2, \infty)$, then $\bar{\rho} \bar{h} \in H^{1, p \wedge q}(B)$.
\end{itemize}
\end{lem}
\noindent
\begin{proof}
(i) Let $\bar{\rho} \in H^{1,p}(B)$ with $p \in (d, \infty)$ and $\bar{h} \in H_0^{1,2}(B)$, Then, by \cite[Theorem 4.3]{EG15}
there exist sequences $(\bar{\rho}_n)_{n \geq 1} \subset C^{\infty}(\overline{B})$ and $(\bar{h}_n)_{n \geq 1} \subset C_0^{\infty}(B)$ such that $\lim_{n \rightarrow \infty} \bar{\rho}_n =\bar{\rho}$ in $H^{1,p}(B)$ and $\lim_{n \rightarrow \infty} \bar{h}_n = \bar{h}$ in $H^{1,2}(B)$.  Observe that $\bar{\rho}_n \bar{h}_n \in C_0^{\infty}(B)$ and
$$
\nabla (\bar{\rho}_n \bar{h}_n) = \bar{h}_n \nabla \bar{\rho}_n + \bar{\rho}_n \nabla \bar{h}_n.
$$
Since $\lim_{n \rightarrow \infty} \bar{h}_n = \bar{h}$ in $L^{\frac{2d}{d-2}}(B)$, $\lim_{n \rightarrow \infty} \nabla \bar{\rho}_n = \nabla \bar{\rho}$ in $L^d(B, \mathbb{R}^d)$, and $\lim_{n \rightarrow \infty} \bar{\rho}_n = \bar{\rho}$ in $L^{\infty}(B)$, $\lim_{n \rightarrow \infty} \nabla \bar{h}_n = \nabla \bar{h}$ in $L^2(B, \mathbb{R}^d)$, we obtain that 
\begin{align*}
\lim_{n \rightarrow \infty} \nabla (\bar{\rho}_n \bar{h}_n)  = \bar{h}  \nabla \bar{\rho} + \bar{\rho} \nabla \bar{h} = \nabla (\bar{\rho} \bar{h}), \quad \text{ in $L^2(B, \mathbb{R}^d)$}.
\end{align*}
Similarly, we have $\lim_{n \rightarrow \infty} \bar{\rho}_n \bar{h}_n  = \bar{\rho} \bar{h}$ in $L^2(B)$. Therefore, $\bar{\rho} \bar{h} \in H^{1,2}_0(B)$. \\ \\
(ii) Let $\bar{\rho} \in H^{1,p}(B)$ with $p \in (d, \infty)$ and $\bar{h} \in H^{1,q}(B)$ with $q \in [2, \infty)$. Then, $\bar{\rho} \bar{h} \in H^{1,1}(B)$ and
$$
\nabla (\bar{\rho} \bar{h}) = \bar{h} \nabla \bar{\rho} + \bar{\rho} \nabla \bar{h}.
$$
Note that $\bar{\rho} \nabla \bar{h} \in L^q(B)$. Now consider the following cases: \\ \\
Case 1): If $q \in [2, d)$, then $\bar{h} \in L^{\frac{qd}{d-q}}(B)$. Since $\nabla \bar{\rho} \in L^d(B, \mathbb{R}^d)$, we get $\bar{h} \nabla \bar{\rho} \in L^q(B, \mathbb{R}^d)$. \\
Case 2): If $q = d$, then $\bar{h} \in H^{1,d}(B) \subset L^{\frac{dp}{p-d}}(B)$. Since $\nabla \bar{\rho} \in L^p(B, \mathbb{R}^d)$, we have $\bar{h} \nabla \bar{\rho} \in L^d(B, \mathbb{R}^d)$. \\
Case 3): If $q \in (d, \infty)$, then $\bar{h} \in H^{1,q}(B) \subset L^{\infty}(B)$. Thus, $\bar{h} \nabla \bar{\rho} \in L^p(B, \mathbb{R}^d)$. \\ \\
From Cases 1--3, we conclude that $\nabla (\bar{\rho} \bar{h}) \in L^{p \wedge q}(B, \mathbb{R}^d)$. Since $\bar{\rho} \bar{h} \in L^q(B) \subset L^{p \wedge q}(B)$, we deduce that $\bar{\rho} \bar{h} \in H^{1, p \wedge q}(B)$.
\end{proof}


\section{Conclusions and discussion} \label{condisc}
This paper addresses the local elliptic regularity results for stationary Fokker-Planck equations with general coefficients, demonstrating that when a solution with low regularity is given a priori, its density belongs to $H^{1,2}_{\text{loc}}(\mathbb{R}^d) \cap C(\mathbb{R}^d)$ and, more precisely, to $H^{1,p}_{\text{loc}}(U) \cap C^{0,1-d/p}_{\text{loc}}(U)$ with $p \in (d, \infty)$ (see Theorem \ref{intromainth}). 
The main results in this manuscript are derived using sectorial Dirichlet forms and their corresponding resolvents.
This main results enable us to conduct a qualitative analysis of the invariant measures of solutions to stochastic differential equations (SDEs) and ensures that the invariant measures can be characterized as a solution to divergence-type partial differential equations (PDEs). Consequently, these findings hold significant potential for diverse applications in numerical analysis.\\
A detailed study and calculations are required to establish the counterpart of the results in this paper for the case \( d = 2 \). Throughout the paper, the $VMO$ condition on $A$ has been imposed, and it is necessary to investigate whether this condition can be removed or relaxed. Further research is required to lower the regularity assumptions on the drift vector fields and the zero-order terms. Specifically, further investigation is needed to explore the generalizability of the condition $c \in L^{p}_{\text{loc}}(U)$ in Corollary \ref{finmaincor}. Moreover, relaxing the regularity assumption on the drifts $\tilde{\mathbf{H}} \in L^p_{loc}(U, \mathbb{R}^d)$ 
in Corollary \ref{finmaincor} presents an interesting direction for future research.\\
The existence of solutions to the stationary Fokker-Planck equation, i.e. the existence of infinitesimally invariant measures, is another critical topic. For instance, in \cite[Theorem 1(i)]{BRS12} (cf. \cite[Chapter 2]{BKRS15} and \cite[Section 2.2]{LST22}), the existence of infinitesimally invariant measures was demonstrated using elliptic PDE theory, particularly the weak maximum principle and Harnack inequality. Developing analytical methods to ensure the existence of solutions under various conditions remains a necessary task. Alternatively, the existence of solutions to the stationary Fokker-Planck equation could be investigated stochastically by examining the invariant probability measure or limiting probability measure of the corresponding diffusion process (cf. \cite{K12}). Further studies are also needed to identify conditions under which the solution to the stationary Fokker-Planck equation represents a probability measure, especially in relation to the stationary probability measure of the diffusion process (cf. \cite[Theorem 2]{BRS12} and \cite[Corollary 2.4.2]{BKRS15}).
\\
Based on the local elliptic regularity results presented in this paper, exploring the possibility of establishing new results on the uniqueness of infinitesimally invariant measures is a promising direction for future research. A research dealing with uniqueness of infinitesimally invariant measures analytically is presented in \cite{BRS02} and \cite[Chapter 4]{BKRS15}. In particular, the result on non-uniqueness is referred to in \cite{Sh08}. Notably, in \cite{LT22}, it was demonstrated that the recurrence of the corresponding diffusion process or semigroup implies the uniqueness of the infinitesimally invariant measure. In that work, the local regularity results in \cite{BKR01} for solutions to stationary Fokker-Planck equations played a crucial role. Although the uniqueness problem could be addressed purely analytically, adopting a stochastic approach provides a compelling and rich framework for further investigation.
\centerline{}
\centerline{}
\centerline{}
\noindent
{\bf Acknowledgment.}\;
The author sincerely appreciates the thoughtful comments and valuable feedback provided by the anonymous referees on this paper.

\centerline{}
\centerline{}
Haesung Lee\\
Department of Mathematics and Big Data Science,  \\
Kumoh National Institute of Technology, \\
Gumi, Gyeongsangbuk-do 39177, Republic of Korea, \\
E-mail: fthslt@kumoh.ac.kr, \; fthslt14@gmail.com
\end{document}